\documentclass[12pt, oneside, psamsfonts]{amsart}

\newif\ifPDF
\ifx\pdfoutput\undefined\PDFfalse
\else \ifnum \pdfoutput > 0 \PDFtrue
        \else \PDFfalse
        \fi
\fi

\usepackage[centertags]{amsmath}
\usepackage{amsfonts}
\usepackage{mathrsfs}
\usepackage{textcomp}
\usepackage{amssymb}
\usepackage{amsthm}
\usepackage{newlfont}
\usepackage[all]{xy}

\ifPDF
  \usepackage[pdftex]{color, graphicx}
  \usepackage[pdftex, bookmarks, colorlinks]{hyperref}
  \hypersetup{colorlinks=false}

\else
  \usepackage{color}
  \usepackage[dvips]{graphicx}
  \usepackage[dvips]{hyperref}
\fi

\usepackage[scale=0.8]{geometry}

\newtheorem{thm}{Theorem}[section]
\newtheorem{cor}[thm]{Corollary}
\newtheorem{lem}[thm]{Lemma}
\newtheorem{prop}[thm]{Proposition}
\theoremstyle{definition}
\newtheorem{defn}[thm]{Definition}
\theoremstyle{remark}
\newtheorem{rem}[thm]{Remark}
\newtheorem{example}[thm]{Example}
\numberwithin{equation}{section}

\newtheorem{NN}[thm]{}

\newcommand{\norm}[1]{\left\Vert#1\right\Vert}
\newcommand{\abs}[1]{\left\vert#1\right\vert}

\newcommand{\Real}{\mathbb R}

\newcommand{\Comp}{\mathbb C}

\newcommand{\eps}{\varepsilon}

\newcommand{\F}{\mathcal{F}}
\newcommand{\Kzero}{\textrm{K}_0}

\newcommand{\tr}{\textrm{T}}
\newcommand{\MC}[2]{\mathrm{M}_{#1}(\mathrm{C}(#2))}

\begin{document}

\title{Mean dimension and AH-algebras with diagonal maps}

\author{Zhuang Niu}
\address{Department of Mathematics and Statistics, Memorial University of Newfoundland, St. John's, NL, Canada~\ A1C5S7}
\curraddr{Department of Mathematics, University of Wyoming, Laramie, WY, USA~\ 82071}
\email{zniu@uwyo.edu}

%\thanks{}
\subjclass[2010]{46L35}
\keywords{AH-algebras, mean dimension, dimension growth, Cuntz semigroup.}
%\date{\today}
%\dedicatory{}
%\commby{}

%------------------------------------------abstract--------------------------------------

\begin{abstract}
Mean dimension for AH-algebras with diagonal maps is introduced. It is shown that if a simple unital AH-algebra with diagonal maps has mean dimension zero, then it has strict comparison on positive elements. In particular, the strict order on projections is determined by traces. Moreover, a lower bound of the mean dimension is given in term of {Toms'} comparison radius. Using classification results, if a simple unital AH-algebra with diagonal maps has mean dimension zero, it must be an AH-algebra without dimension growth.

Two classes of AH-algebras {with diagonal maps} are shown to have mean dimension zero: the class of AH-algebras with at most countably many extremal traces, and the class of AH-algebras with numbers of extreme traces which induce same state on the $\Kzero$-group being uniformly bounded (in particular, this class includes AH-algebras with real rank zero).
%
%Other versions of mean dimension, which works better for more general AH-algebras, are also introduced.
\end{abstract}

\maketitle

\setcounter{tocdepth}{1}
\tableofcontents

\section{Introduction}
Dimension growth---roughly speaking, the limit of ratios of the dimension of the base space and dimension of the irreducible representation---plays a crucial role in the classification of AH-algebras. On one hand, the class of simple unital AH-algebras without dimension growth is classified by their K-theory invariant (see \cite{EG-RR0AH}, \cite{Gong-AH}, and \cite{EGL-AH}). On the other hand, AH-algebras with fast dimension growth { lack} certain regularities in their invariants (see, for example, \cite{Vill-perf}, \cite{Vill-sr}, \cite{Toms-Ann}).

However, the definition of dimension growth relies on inductive limit decompositions. It is then interesting to study whether one can tell whether a given AH-algebra has slow dimension growth without looking at the inductive sequence itself. For instance, { it is an open question that} whether any real rank zero AH-algebra has slow dimension growth automatically.

In this paper, a positive answer to the question above is given for the class of AH-algebras with diagonal maps (Corollary 4.4) (this class of AH-algebras includes the examples of { Villadsen} in \cite{Vill-perf} and the examples of Toms in \cite{Toms-Ann}). In fact, it is shown that if the tracial simplex of the given AH-algebra with diagonal maps is relatively small, then it must {have} slow dimension growth. More precisely,

\begin{thm}[Corollary \ref{cor-general}]
If a unital simple AH-algebra with diagonal maps has at most countably many extremal tracial states, or if there is $M>0$ such that $\rho^{-1}(\kappa)$ has at most $M$ extreme points for any $\kappa\in\mathrm{S}(\Kzero(A))$, then it is an AH-algebra without dimension growth.
\end{thm}

The main tool in the proof of the theorem above is the mean dimension for AH-algebras, which is {an analog to} the dynamical system mean dimension introduced by Elon Lindenstrauss and Benjamin Weiss in \cite{Lindenstrauss-Weiss-MD}. It is defined as the limit ratio of the dimensions of certain open covers of the base spaces (instead of the dimension of the base spaces themselves) and the dimensions of the irreducible representations (it should be noted that in order to define the mean dimension, the eigenvalue {patterns} of the AH-algebras are assumed to be induced by continuous maps between the base spaces).  As the role played by dimension growth, it also measures the size of the space in which one can maneuver a vector bundle (a projection), but in a much more intrinsic way. 

Like the dimension growth, the definition of the mean dimension relies on the individual inductive limit decomposition. In order to decide whether a certain AH-algebra has mean dimension zero automatically, small boundary properties (Definition \ref{Defn-SBP}, Definition \ref{SBRP}) are also introduced, and it is shown that AH-algebras with these small boundary properties always have mean dimension zero. This enables us to show that any AH-algebra of real rank zero has mean dimension zero; see Theorem \ref{RR0-SBP}. (It is worth to point out that, unlike mean dimension and small boundary property for minimal dynamical systems, which were shown to be equivalent (Theorem 6.2 of \cite{Lind-MD}), there are simple AH-algebras with mean dimension zero but without small boundary property (Remark \ref{MD0>SBP}).)

Then if one only focuses on the class of AH-algebras with diagonal maps, mean dimension gives a nice upper bound for the dimension ratio in the following local approximation theorem.

\begin{thm}[Theorem \ref{LOA}]
Let $A$ be a simple AH-algebra with diagonal maps, then $A$ can be locally approximated by homogeneous C*-algebras with dimension ratio no more than its mean dimension. In particular, if this AH-algebra has mean dimension zero, it has local slow dimension growth. 
\end{thm}

Using this local approximation, AH-algebras with mean dimension zero and diagonal maps are shown to have strict comparison on positive elements (Theorem \ref{MD0-Comp}), and hence the strict order on projections are determined by traces (Theorem \ref{K0-comp}). In particular,  if an AH-algebra with diagonal maps has real rank zero, the strict order on projections must be  determined by traces. Combining with a result of Huaxin Lin in \cite{Lin-AHRR0}, real rank zero AH-algebras with diagonal maps are tracially AF algebra, and therefore they are AH-algebras without dimension growth.

As another application of the above-mentioned approximation theorem, the Cuntz semigroups of AH-algebras with diagonal maps and mean dimension zero are calculated. Together with the recent classification results, these AH-algebras are AH-algebras without dimension growth.

The comparison radius of an AH-algebra with diagonal maps is also calculated. As one expects, it is dominated by one half of its mean dimension (but it is not known if this is also the lower bound).

However, this version of mean dimension only works nicely for AH-algebras with diagonal maps. For instance, the Villadsen algebras of the second type (see \cite{Vill-sr}) always have mean dimension zero, but they cannot have slow dimension growth, even locally. Thus, several modified versions of mean dimension---the Cuntz mean dimension for AH-algebras with generalized diagonal maps, and the variation mean dimension for general AH-algebras---are also introduced. These new versions of mean dimension are able to detect the regularity of the K-theoretical invariant in a {broader} context. For example, AH-algebras with zero variation mean dimension always have strict comparison on positive elements. But not as the mean dimension for AH-algebra with diagonal maps, {it seems that} these versions of mean dimension are very difficult to calculate, and in particular, it is unknown if real rank zero would imply {mean dimension zero} in these contexts. However, for the class of AH-algebras with diagonal maps, all of these definitions agree with each other.

\subsection*{Acknowledgement} The author is greatly indebted to Professor N.~C.~Phillip for informing him the construction of Example \ref{AH-Model} and drawing his attention to the mean dimension during the AIM workshop on Cuntz semigroups in November 2009. This eventually leads to the main work in this paper. The author also thanks Professor G.~A.~Elliott for the discussions during his visit to Memorial University of Newfoundland. The work is supported by a Natural Sciences and Engineering Research Council of Canada (NSERC) Discovery Grant.

\section{Notation and preliminaries}

\begin{defn}\label{CuntzR}
Let $a$ and $b$ be positive elements in a C*-algebra $A$. The element $a$ is said to be Cuntz smaller than {$b$}, denoted by $a\precsim b$, if there exist sequences $(x_n)$ and $(y_n)$ in $A$ such that $$\lim_{n\to\infty} x_nby_n=a.$$ The elements $a$ and $b$ are {said to be} Cuntz equivalent, denoted by $a\sim b$, if $a\precsim b$ and $b\precsim a$. It is an {equivalence} relation.
\end{defn}

{In fact, the elements $x_n$ and $y_n$ in Definition \ref{CuntzR} can be chosen so that $x_n=y_n^*$.}  Moreover, for any positive element $a\in A$ and any $\eps>0$, denoted by $(a-\eps)_+$ the element $h_\eps(a)$, where $h_\eps(t)=\min\{0, t-\eps\}$. One then has the following {facts} on comparison of positive elements $a$ and $b$.
\begin{enumerate}
%\item $a\precsim b$ {if and only if} there are $\{c_n\}$ such that $c_n^*bc_n\to a$.
\item $a\precsim b$ if and only if $(a-\eps)_\eps\precsim b$ for all $\eps>0$.
\item $a\precsim b$ if and only if for any $\eps>0$, there exists $x$ such that $x^*bx=(a-\eps)_\eps$.
\item if $\norm{a-b}<\eps$, then $(a-\eps)_+\precsim b$.
\item $a+b\precsim a\oplus b$ {in $\mathrm{M}_2(A)$}; if $a\perp b$, then $a+b\sim a\oplus b$.
\end{enumerate}
The proofs can be found in \cite{Cuntz-CuGr} and \cite{RorUHF-II}.

\begin{defn}
Let $A$ be a C*-algebra, and consider matrix algebras $\textrm{M}_\infty(A):=\bigcup_n\mathrm{M}_n(A)$ over $A$. {Let $a\in\mathrm{M}_n(A)$ and $b\in\mathrm{M}_m(A)$ be positive elements. Then $a$ is said to be Cuntz smaller than $b$, still denoted by $a\precsim b$, if there are $\{c_n\}\subseteq \mathrm{M}_{m, n}(A)$ such that $c^*_nbc_n\to a$; and $a$ is said to be equivalent to $b$, denoted by $a\sim b$, if $a\precsim b$ and $b\precsim a$}. Then $\sim$ is an equivalent relation on $A_+$. Denoted by $[a]$ the {equivalence} class containing $a$. Then $W(A):=\{[a];\ a\in \textrm{M}_\infty(A)^+\}$ is a positive ordered abelian semigroup with $$[a]+[b]:=[a\oplus b]\quad\textrm{and}\quad [a]\leq[b]\ \textrm{if}\ a\precsim b.$$
\end{defn}

\begin{defn}
Let $a$ be a positive element in a C*-algebra $A$, and let $\tau$ be a state of $A$. Define $$d_{\tau}(a):=\lim_{\eps\to0^+}\tau(f_\eps(a))=\sup_{\eps>0}\tau(f_{\eps}(a))=\mu_\tau((0,\norm{a}]),$$ where 
$f_{\eps}(x):=\min\{x/\eps, 1\}$, and $\mu_\tau$ is the measure induced by $\tau$ on $\textrm{C*}(1, a)\cong\mathrm{C}(\mathrm{sp}(a))$.

{Assume the C*-algebra $A$ is simple, then it} is said to have strict comparison of positive elements if for any positive elements $a$ and $b$ with $$d_\tau(a)< d_\tau(b),\quad\forall \tau\in\mathrm{T}(A),$$ one has that $a\precsim b$.
\end{defn}

Note that the dimension function $d_\tau$ induces a positive map from $W(A)$ to $\mathrm{SAff}(\textrm{T}(A))$, the set of functions on $\textrm{T}(A)$ which are pointwise suprema of increasing sequences of continuous, affine, and strictly positive functions on $\textrm{T}(A)$, equipped with the natural order (\cite{Brn-Toms-3app}). If the C*-algebra $A$ has strict comparison of positive element, then the order structure on $W(A)$ is determined by its image in $\mathrm{SAff}(\textrm{T}(A))$.

\begin{defn}\label{diag-map-defn}
An AH-algebra is an inductive limit of C*-algebras 
\begin{equation}\label{HomBloc}
p\MC{n}{X}p
\end{equation} 
where $X$ is a compact Hausdorff space, and $p$ is a projection in $\MC{n}{X}$. {By \cite{Bl-AH}, any AH-algebra can be written as an inductive limit of C*-algebras in the form of \eqref{HomBloc} with $X$ a CW-complex. Therefore, by written $X$ as a disjoint union of its connected components $X_{1}, X_2, ..., X_h$, one has that any AH-algebra is an inductive limit of the C*-algebras in the form of $$\bigoplus_{l=1}^{h}p_{l}\MC{r_{l}}{X_{l}}p_{l},$$ where each $X_l$ is a connected CW-complex and $p_l$ is a projection in $\textrm{M}_{r_l}(\textrm{C}(X_l))$.}

Consider an AH-algebra $A=\varinjlim_{n\to\infty}(A_i, \varphi_i)$ {with} $A_i=\bigoplus_{l=1}^{h_i}p_{i, l}\MC{r_{i,l}}{X_{i, l}}p_{i, l},$ where each {$X_{i, l}$ is connected}. Denote by $A_{i, l}$ the direct summand $p_{i, l}\MC{r_{i,l}}{X_{i, l}}p_{i, l}$. For any $i, j$, if there exist a unitary $U\in A_j$ such that the restriction of the map $\mathrm{ad}(U)\circ\varphi_{i, j}$ to any direct summands $A_{i, l}$ and $A_{j, k}$ has the form $$f\mapsto\left(
\begin{array}{ccc}
f\circ\lambda_1 & &\\
&\ddots&\\
 && f\circ\lambda_n
\end{array}
\right)
$$
for some continuous maps $\lambda_1,...,\lambda_n : X_{j, k}\to X_{i, l}$, then $A$ is called an AH-algebra with diagonal maps.
\end{defn}

{The following example was informed to the author by N.~C.~Phillips in a private communication during an AIM workshop in 2009. This example is particularly important, and it serves as the main motivation for the work in this paper. This example is also a special case of the construction given by Blackadar and Kumjian in 1.1 of \cite{BlaKum-Simple}.}
\begin{example}[AH-model for dynamical systems]\label{AH-Model}
Let $X$ be a compact Hausdorff space, and let $\sigma$ be a homeomorphism of $X$. The the AH-model of the system $(X, \sigma)$ is the following AH-algebra $A(X, \sigma)$ with diagonal map:
\begin{displaymath}
\xymatrix{
\mathrm{C}(X)\ar[r] & \MC{2}{X} \ar[r] & \MC{2^2}{X} \ar[r] & \cdots
}
\end{displaymath}
with the connecting map induced by$$f\mapsto\left(
\begin{array}{cc}
f &\\
 & f\circ\sigma
\end{array}
\right).
$$

{By 1.3 of \cite{BlaKum-Simple}, the AH-algebra $A(X, \sigma)$ is simple if $(X, \sigma)$ is minimal. On the other hand, if there is a closed nontrivial invariant subset of $X$, it clearly induces a nontrivial closed two-sided ideal of $A(X, \sigma)$.  Moreover, the tracial simplex of $A(X, \sigma)$ is canonically isomorphic to the simplex of the invariant probability measures on $X$.}
\end{example}

\begin{defn}
Let $\displaystyle{A=\varinjlim_{i\to\infty}(A_i, \varphi_i)}$ be an AH-algebra with $\displaystyle{A_i=\bigoplus_{l=1}^{h_i}p_{i, l}\MC{r_{i,l}}{X_{i, l}}p_{i, l}}$, where each $p_{i, l}$ has constant rank. Then the dimension growth of $A$ is defined by $$\liminf_{i\to\infty}\max_l\{\frac{\mathrm{dim}(X_{i, l})}{\mathrm{rank}(p_l)}\},$$ where $\mathrm{dim}(X_{i, l})$ is the covering dimension of $X$. If there is an inductive decomposition of $A$ {with} homogeneous C*-algebras such that the dimension growth is zero, then $A$ is said to have slow dimension growth. Moreover, if the dimensions of the base spaces are uniformly bounded, $A$ is said to be an AH-algebra without dimension growth.
\end{defn}

\begin{rem}
The class of simple unital real rank zero AH-algebra with slow dimension growth has been classified in \cite{EG-RR0AH}, and the class of simple AH-algebra without dimension growth has been classified in \cite{Gong-AH} and \cite{EGL-AH}. {Together with the recent progress in the classification theory (see \cite{Winter-Z}, \cite{Winter-Z-stable-01}, \cite{Winter-Z-stable-02}, \cite{Lnclasn}, \cite{L-N}), it was shown in  \cite{Toms-SDG} that the class of simple unital AH-algebras with slow dimension growth coincides with the class of simple unital AH-algebras without dimension growth.}
\end{rem}

Let $A$ be a unital C*-algebra and let $\tau$ be a tracial state. Recall that the restriction of $\tau$ to projections induces a positive linear functional on $\Kzero(A)$, and this in fact induces an affine map $$\rho: \mathrm{T}(A)\to\mathrm{S}_u(\Kzero(A)),$$ where $\mathrm{S}_u(\Kzero(A))$ is the convex of positive linear functional on $\Kzero(A)$ which send $[1_A]_0$ to $1$. 

{It follows from \cite{BlaTrace} that any element of $\mathrm{S}_u(\Kzero(A))$ comes from a quasi-trace of $A$; and if $A$ is exact, it was shown by Haagerup (\cite{HaagTrace}) that any quasi-trace of $A$ is in fact a trace. Therefore, the map $\rho$ is always surjective when $A$ is exact.}

Note that $\rho$ is injective if and only if projections of $A$ is separated by traces. Recall that a C*-algebra $A$ is real rank zero if {any self-adjoint element of $A$ can be approximated by self-adjoint elements with finite spectrum. In particular, this implies that the linear subspace spanned by the projections of $A$ is dense, and then it is evident that if $A$ has real rank zero, the projections of $A$ separate traces.}

\section{Mean dimension for AH-algebras}

Mean dimension was introduced by E. Lindenstrauss and B. Weiss in \cite{Lindenstrauss-Weiss-MD} for {topological} dynamical systems. In this section, AH-algebras with eigenvalue {patterns induced} by continuous maps are considered, and the mean dimension will be introduced for these AH-algebras.

Let $X$ be a topological space, and let $\alpha$ be a finite open cover of $X$. We say that a cover $\beta$ refines $\alpha$ ($\beta\succ\alpha$) if every element of $\beta$ is a subset of some element of $\alpha$. Denote the set of finite open covers of $X$ by $\mathcal{C}(X)$.

\begin{defn}[Definition 2.1 and Definition 2.2 of \cite{Lindenstrauss-Weiss-MD}]
If $\alpha$ is an open cover of $X$, denote by $$\mathrm{ord}(\alpha):=\max_{x\in X}(\sum_{U\in\alpha}1_{U}(x))-1,$$ and denote by $$\mathcal D(\alpha)=\min_{\beta\succ\alpha}\mathrm{ord}(\beta).$$

A continuous map $f: X\to Y$ is $\alpha$-compatible if it is possible to find a finite open cover of $f(X)$, $\beta$, such that $f^{-1}(\beta)\succ\alpha.$ 
\end{defn}

\begin{rem}
It follows from Proposition 2.3 of \cite{Lindenstrauss-Weiss-MD} that a map $f: X\to Y$ is $\alpha$-compatible if for any $y\in Y$, $f^{-1}(y)$ is a subset of some $U\in\alpha$. 
\end{rem} 

The following proposition is a characterization of $\mathcal{D}(\alpha)$.
\begin{prop}[Proposition 2.4 of \cite{Lindenstrauss-Weiss-MD}]\label{comp-map}
If $\alpha$ is an open cover of $X$, then $$\mathcal{D}(\alpha)\leq k$$ if and only if there is an $\alpha$-compatible continuous function $f: X\to K$ where $K$ has topological dimension $k$.
\end{prop}

{If $\alpha$ and $\beta$ are two open covers of $X$, then let $\alpha\vee\beta$ denote the open cover of $X$ consisting of $U \cap V$ where $U\in\alpha$ and $V\in\beta$.} It follows from Corollary 2.5 of \cite{Lindenstrauss-Weiss-MD} that
\begin{eqnarray*}
\mathcal{D}(\alpha\vee\beta)&\leq&\mathcal{D}(\alpha)+\mathcal{D}(\beta).
\end{eqnarray*}

\begin{lem}\label{dec}
Let $\alpha$ be an open cover of $Y$, and let $f: X\to Y$ be a continuous map. Then $$\mathcal{D}(f^{-1}(\alpha))\leq\mathcal{D}(\alpha).$$
\end{lem}
\begin{proof}
Let $\beta$ be an open cover with $\beta\succ\alpha$ with $\mathrm{ord}(\beta)=\mathcal{D}(\alpha)$. Consider the open cover $f^{-1}(\beta)$. It is clear that $f^{-1}(\beta)\succ f^{-1}(\alpha)$. Pick $x\in X$ such that $$\mathrm{ord}(f^{-1}(\beta))=\sum_{V\in f^{-1}(\beta)} 1_{V}(x)-1.$$ However, since
\begin{eqnarray*}
\sum_{V\in f^{-1}(\beta)} 1_{V}(x)&=&\sum_{U\in \beta} 1_{U}(f(x)),
\end{eqnarray*}
one has
\begin{eqnarray*}
\mathcal{D}(f^{-1}(\alpha))&\leq&\mathrm{ord}(f^{-1}(\beta))\\
&=&\sum_{V\in f^{-1}(\beta)} 1_{V}(x)-1\\
&=&\sum_{U\in \beta} 1_{U}(f(x))-1\\
&\leq&\mathrm{ord}(\alpha),
\end{eqnarray*}
as desired.
\end{proof}

Consider the unital homogeneous C*-algebras $$A_i=\bigoplus_{l=1}^{h_i}p_{i, l}\MC{r_{i,l}}{X_{i, l}}p_{i, l},$$ {where each $X_{i, l}$ is a connected compact Hausdorff space and each $p_{i, l}$ is a projection in $\MC{r_{i,l}}{X_{i, l}}$}. Denote by $n_{i, l}$ the rank of $p_{i, l}$, {and denote} by $A_{i, l}$ the $l$th component of $A_i$. Consider the family of homomorphisms $\varphi_{i_1, i_3}: {A_{i_1}\to A_{i_3}}$ with $\varphi_{i_1, i_3}=\varphi_{i_2, i_3}\circ\varphi_{i_1, i_2}$ for any $i_1\leq i_2\leq  i_3$. {Denote by $m_{i, j}^{l, k}$ the multiplicity of the restriction of the map $\varphi_{i, j}$ to $A_{i, l}$ and $A_{j, k}$, and note that 
$$n_{j, k}=\sum_{l=1}^{h_i}m_{i, j}^{l, k}n_{i, l}.$$}

{As in \cite{Gong-AH}, for any homomorphism $\varphi: A\to B$, where $A$ and $B$ are homogeneous C*-algebras, and for any $x$ in the spectrum of $B$, let $\mathrm{SP}\varphi_y$ denote the subset of the spectrum of $A$ which corresponds to $\pi_x\circ\varphi$, where $\pi_x$ is the evaluation map at $x$.
Then, in this paper, let us assume that the eigenvalue pattern $$\Gamma_{i, j}^{k, l}: X_{j, k}\ni x\mapsto \mathrm{SP}(\varphi_{i, j}|_{A_{i, l}})_x \in \prod_{l=1}^{m_{i, j}^{l, k}} X_{i, l}/\sim$$ of the restriction of $\varphi_{i, j}$ to $A_{i, l}$ and $A_{j, k}=p_{j, k}\MC{r_{j,k}}{X_{j, k}}p_{j, k}$ is alway induced by the continuous maps $\{\lambda_{i, j}^{l, k}(m);\ 1\leq m\leq m_{i, j}^{l,k} \}$, that is $$\mathrm{SP}(\varphi_{i, j}|_{A_{i, l}})_x=\{\lambda_{i, j}^{l, k}(1)(x), ..., \lambda_{i, j}^{l, k}(m_{i, j}^{l, k})(x)\},\quad\forall x.$$}

Consider the unital inductive limit of {$(A_i, \varphi_i)$}:
\begin{displaymath}
\xymatrix{
A_1\ar[r]^-{\varphi_{1,2}}&A_2\ar[r]^-{\varphi_{2, 3}}&\cdots\ar[r]&A=\varinjlim A_i.
}
\end{displaymath}

Let $\alpha_{i, l}$ be an open cover of $X_{i, l}$. Consider the open cover $$(\lambda_{i, j}^{l, k}(1))^{-1}(\alpha_{i, l})\vee\cdots\vee (\lambda_{i, j}^{l, k}(m_{i, j}^{l, k}))^{-1}(\alpha_{i, l}),$$ and denoted by $\varphi_{i, j}^{l, k}(\alpha_{i, l})$.

Let $\alpha_i$ be an open cover of $X_i$ with each member a subset of a connected component. Denote by $\alpha_{i, l}$ the cover of $X_{i, l}$ induced by $\alpha_i$. Then set $$\varphi_{i, j}^k(\alpha_i)=\varphi_{i, j}^{1, k}(\alpha_{i, 1})\vee\cdots\vee \varphi_{i, j}^{h_i, k}(\alpha_{i, h_i}).$$

Consider $\mathcal{D}(\varphi_{i, j}^k(\alpha_{i}))$. One then has
\begin{eqnarray*}
\mathcal{D}(\varphi_{i, j}^k(\alpha_i))&=&\mathcal{D}(\varphi_{i, j}^{1, k}(\alpha_{i, 1})\vee\cdots\vee \varphi_{i, j}^{h_i, k}(\alpha_{i, h_i}))\\
&\leq&\mathcal{D}(\varphi_{i, j}^{1, k}(\alpha_{i, 1}))+\cdots +\mathcal{D}(\varphi_{i, j}^{h_i, k}(\alpha_{i, h_i}))\\
&=&\sum_{l=1}^{h_i}\mathcal{D}((\lambda_{i, j}^{l, k}(1))^{-1}(\alpha_{i, l})\vee\cdots\vee (\lambda_{i, j}^{l, k}(m_{i, j}^{l, k}))^{-1}(\alpha_{i, l}))\\
&\leq&\sum_{l=1}^{h_i}\sum_{m=1}^{m_{i, j}^{l, k}}\mathcal{D}((\lambda_{i, j}^{l, k}(m))^{-1}(\alpha_{i, l})\\
&\leq&\sum_{l=1}^{h_i}\sum_{m=1}^{m_{i, j}^{l, k}}\mathcal{D}(\alpha_{i, l})\quad\textrm{by Lemma \ref{dec}}\\
&=&\sum_{l=1}^{h_i}m_{i, j}^{l, k}\mathcal{D}(\alpha_{i, l}),
\end{eqnarray*}
and hence

\begin{equation}\label{0001}
\frac{\mathcal{D}(\varphi_{i, j}^k(\alpha_i))}{n_{j, k}}\leq\frac{\sum_{l=1}^{h_i}m_{i, j}^{l, k}\mathcal{D}(\alpha_{i, l})}{n_{j, k}}
=\frac{\sum_{l=1}^{h_i}m_{i, j}^{l, k}\mathcal{D}(\alpha_{i, l})}{\sum_{l=1}^{h_i}m_{i, j}^{l, k}n_{i, l}}.
\end{equation}

\begin{lem}
Let $a_1,...,a_n$, $b_1,..., b_n$, $m_1, ..., m_n$ be positive numbers with $b_i$ and $m_i$ nonzero. If $a_i/b_i\leq c$ for some number $c$ for any $1\leq i\leq n$, then 
$$\frac{m_1a_1+\cdots +m_na_n}{m_1 b_1+\cdots +m_nb_n}\leq c.$$
\end{lem}
\begin{proof}
Since $a_i/b_i\leq c$, one has that $a_i\leq cb_i$, and hence $m_ia_i\leq cm_ib_i$ for all $1\leq i\leq n$. Therefore $$m_1a_1+\cdots+m_na_n\leq cm_1b_1+\cdots+cm_nb_n=c(m_1b_1+\cdots+m_nb_n),$$ and $$\frac{m_1a_1+\cdots +m_na_n}{m_1 b_1+\cdots +m_nb_n}\leq c,$$ as desired.
\end{proof}

From the lemma above and Equation \eqref{0001}, one has $$\frac{\mathcal{D}(\varphi_{i, j}^k(\alpha_i))}{n_{j, k}}\leq\max\{\frac{\mathcal{D}(\alpha_{i, l})}{n_{i, l}};\ 1\leq l\leq h_i\}.$$ Thus, the sequence $$\max\{\frac{\mathcal{D}(\varphi_{i, j}^k(\alpha_{i}))}{n_{j, k}};\ 1\leq k\leq h_j\},\quad j=i, i+1, ...$$ is a decreasing sequence, and the limit exists.

\begin{defn}\label{mdim}
Set $$\gamma_i(A):=\sup_{\alpha_i\in\mathcal{C}(X_i)}\lim_{j\to\infty}\max\{\frac{\mathcal{D}(\varphi_{i, j}^k(\alpha_{i}))}{n_{j, k}};\ 1\leq k\leq h_j\}.$$ Note that $\{\gamma_n(A)\}$ is an increasing sequence. The mean dimension of the inductive limit $A$ is the limit $$\gamma(A)=\lim_{i\to\infty}\gamma_i(A).$$
\end{defn}

\begin{rem}
{Consider the AH-algebras in Example \ref{AH-Model}. They are AH-algebras with diagonal maps, and the eigenvalue maps are $\{\sigma^0, \sigma^1, ..., \sigma^{2^n}\}$ for some $n$. It follows from the definition directly that the mean dimension of the AH-model $A(X, \sigma)$ is equal to the dynamical mean dimension of $(X, \sigma)$ introduced by Lindenstrauss and Weiss in \cite{Lindenstrauss-Weiss-MD}. The author is indebted to N.~C.~Phillips for informing him the construction of Example \ref{AH-Model}, which leads to the definition above and eventually leads to the main work in this paper.}
\end{rem}

\begin{rem}\label{AH-SD-MD0}
Since $\mathcal{D}(\varphi_{i, j}^k(\alpha_{i}))\leq\mathrm{dim}(X_{j, k})$, it is clear that if an AH-algebra has slow dimension growth is zero, then it has mean dimension zero.
\end{rem}

Let $f$ be a positive element in $A_i$. Then for any $x\in X_{j, k}$, one has
\begin{eqnarray*}
\mathrm{Tr}(\varphi_{i, j}^k(f)(x))&=&\sum_{l=1}^{h_i}\sum_{m=1}^{m_{i, j}^{l, k}}\mathrm{Tr}(f(\lambda_{i, j}^{l, k}(m)(x)))\\
&\leq &\sum_{l=1}^{h_i}\sum_{m=1}^{m_{i, j}^{l, k}}\sup_{y\in X_{i, l}}\mathrm{Tr}(f(y))\\
&=&\sum_{l=1}^{h_i}m_{i, j}^{l, k}\sup_{y\in X_{i, l}}\mathrm{Tr}(f(y)).
\end{eqnarray*}
Hence $$\sup_{x\in X_{j, k}}\mathrm{Tr}(\varphi_{i, j}^k(f)(x))\leq \sum_{l=1}^{h_i}m_{i, j}^{l, k}\sup_{y\in X_{i, l}}\mathrm{Tr}(f(y)),$$ and

\begin{eqnarray*}
\frac{1}{n_{j, k}}\sup_{x\in X_{j, k}}\mathrm{Tr}(\varphi_{i, j}^k(f)(x))&\leq&\frac{\sum_{l=1}^{h_i}m_{i, j}^{l, k}\sup_{y\in X_{i, l}}\mathrm{Tr}(f(y))}{\sum_{l=1}^{h_i}m_{i, j}^{l, k}n_{i, l}}\\
&\leq&\max_{1\leq l \leq h_i}\frac{1}{n_{i, l}}\sup_{y\in X_{i, l}}\mathrm{Tr}(f(y)).
\end{eqnarray*}
Thus, the sequence $$\{\max_{1\leq k\leq h_j} \sup_{x\in X_{j, k}}\frac{1}{n_{j, k}}\mathrm{Tr}(\varphi_{i, j}^k(f)(x)),\ \quad j=i, i+1, ...\}$$ is decreasing, and the limit exists.

\begin{defn}
Let $f\in (A_{i, l})^+$, and let $E$ be a closed subset of $X_{i, l}$. For any $X_{j, k}$ with $j\geq i$, the orbit capacities of $f$ and $E$ at $X_{j, k}$, denoted by $\textrm{ocap}_{j, k}(f)$ and $\textrm{ocap}_{j, k}(E)$ respectively, are 
$$\mathrm{ocap}_{j, k}(f):=\sup_{x\in X_{j, k}}\frac{1}{n_{j, k}}\mathrm{Tr}(\varphi_{i, j}^k(f)(x)),$$
and 
$$\mathrm{ocap}_{j, k}(E):=\sup_{x\in X_{j, k}}\frac{n_{i, l}}{n_{j, k}}\#\{1\leq m\leq m_{i, j}^{l, k};\ \lambda_{i, j}^{l, k}(m)(x)\in E\}.$$

The orbit capacity of $f$ and $E$, denote by $\mathrm{ocap}(f)$ and $\mathrm{ocap}(E)$ respectively, is 
$$\mathrm{ocap}(f):=\lim_{j\to\infty}\max_{1\leq k\leq h_j}\textrm{ocap}_{j, k}(f),$$ and  
$$\mathrm{ocap}(E):=\lim_{j\to\infty}\max_{1\leq k\leq l_j}\mathrm{ocap}_{j, k}(E).$$
%
%For any closed set $E\in X_{i, l}$, the orbit capacity of $E$ is defined by 
%$$\mathrm{ocap}(E):=\lim_{j\to\infty}\max_{1\leq k\leq l_j}\sup_{x\in X_{j, k}}\frac{n_{i, l}}{n_{j, k}}\#\{1\leq m\leq m_{i, j}^{l, k};\ \lambda_{i, j}^{l, k}(m)(x)\in E\cap X_{i, l}\}.$$
\end{defn}

\begin{rem}
For any closed set $E\subseteq X_{i, l}$, one has $$\mathrm{ocap}(E)\leq\inf\{\mathrm{ocap}(f);\ f\ \textrm{is a positive element in the center of $A_i$ and $f|_{E}={p_{i, l}}$} \}.$$
{Indeed, let $f$ be an arbitrary positive element in the center of $A_i$, then $f=\alpha \cdot p_{i, l}$ for some continuous function $\alpha: X_i\to [0, +\infty)$ (since any central element is sent to a scalar multiple of the identity under any irreducible representation). Assuming that the restriction of $f$ to $E$ is $p_{i, l}$, then the restriction of $\alpha$ to $E$ is the constant function $1$. Therefore, for any $j, k$, one has
\begin{eqnarray*}
\mathrm{ocap}_{j, k}(f) & = & \sup_{x\in X_{j, k}}\frac{1}{n_{j, k}}\mathrm{Tr}(\varphi_{i, j}^k(f)(x)) \\
&=& \sup_{x\in X_{j, k}}\frac{1}{n_{j, k}} \sum_{l=1}^{h_i}\sum_{m=1}^{m_{i, j}^{l, k}}\mathrm{Tr}(f(\lambda_{i, j}^{l, k}(m)(x))) \\
&\geq & \sup_{x\in X_{j, k}}\frac{\mathrm{rank}(p_{i, l})}{n_{j, k}}\cdot \#\{1\leq m\leq m_{i, j}^{l, k};\ \lambda_{i, j}^{l, k}(m)(x)\in E\}  \\
&= & \mathrm{ocap}_{j, k}(E),
\end{eqnarray*}
and hence $\mathrm{ocap}(E)\leq \mathrm{ocap}(f)$.

}
\end{rem}

\begin{rem}\label{tr-ocap}
For any $f\in (A_{i, l})^+$, one has $$\mathrm{ocap}(f)=\sup_{\tau\in\mathrm{T}(A)}\tau(\varphi_{i, \infty}(f)).$$ Indeed, using a compactness argument,  it is clear that $$\mathrm{ocap}(f)\leq\sup_{\tau\in\mathrm{T}(A)}\tau(\varphi_{i, \infty}(f)).$$ On the other hand, let $\tau$ be a tracial state with $\tau(\varphi_{i, \infty}(f))$ maximum. Let $\{\mu_j\}$ be the sequence of probability measures with $\mu_j$ supported on $X_j$ induced by the restriction of $\tau$. Then,
\begin{eqnarray*}
\tau_{\mu_j}(\varphi_{i, j}(f))&=&\frac{1}{n_{j, 1}}\int_{X_{j, 1}}\mathrm{Tr}(\varphi_{i, j}^1(f))d\mu_j^1+\cdots+\frac{1}{n_{j, h_j}}\int_{X_{j, 1}}\mathrm{Tr}(\varphi_{i, j}^{h_j}(f))d\mu_j^{h_j}\\
&\leq&\frac{1}{n_{j, 1}}\sup_{x\in X_{j, 1}}\mathrm{Tr}(\varphi_{i, j}^1(f)(x))\abs{\mu_j^1}+\cdots+\frac{1}{n_{j, h_j}}\sup_{x\in X_{j, 1}}\mathrm{Tr}(\varphi_{i, j}^{h_j}(f)(x))\abs{\mu_j^{h_j}}\\
&\leq& \max_{1\leq k\leq h_j}\sup_{x\in X_{j, k}}\frac{1}{n_{j, k}}\mathrm{Tr}(\varphi_{i, j}(f)(x)),
\end{eqnarray*}
{where $\mu_j^k$, $1\leq k\leq h_j$, are the restrictions of $\mu_j$ to $X_{j, k}$.}
Hence $$\sup_{\tau\in\mathrm{T}(A)}\tau(\varphi_{i, \infty}(f))=\tau(\varphi_{i, \infty}(f))\leq\mathrm{ocap}(f).$$
\end{rem}

\begin{defn}\label{Defn-SBP}
An AH-algebra $A$ has the small boundary property (SBP) if for any $X_{i, l}$, any point $x\in X_{i, l}$, any open $U\ni x$, there is a neighbourhood $V\subseteq U$ such that $\mathrm{ocap}(\partial V)=0$.
\end{defn}

\begin{rem}
The small boundary property for AH-algebra is an analog to the small boundary property for dynamical systems introduced in \cite{Lindenstrauss-Weiss-MD} by Lindenstrauss and Weiss.
\end{rem}

\begin{rem}\label{rem-sbp}
If an AH-algebra has (SBP), using a compactness argument, one has the following: for any $\eps>0$, any $X_{i, l}$, any point $x\in X_{i, l}$, any open $U\ni x$, there is a neighbourhood $V\subseteq U$ and $\delta>0$ such that $\mathrm{ocap}(\textrm{B}(\partial V, \delta))<\eps$, where $\textrm{B}(\partial V, \delta)$ is the closed $\delta$-neighbourhood of $\partial V$.

Indeed, since $A$ has (SBP), there exists $j>i$ such that $$\textrm{ocap}_{j, k}(\partial V)<\eps.$$ Then for each $x\in X_{j, k}$, there is a neighbourhood $U\ni x$ and $\delta_U>0$ such that
$$\frac{n_{i, l}}{n_{j, k}}\#\{1\leq m\leq m_{i, j}^{l, k};\ \lambda_{i, j}^{l, k}(m)(x)\in \textrm{B}(\partial V, \delta_U)\}<\eps,\quad\textrm{for any $x\in U$}.$$
Since $X_{j, k}$ is compact, there is $\delta_k>0$ such that 
$$\frac{n_{i, l}}{n_{j, k}}\#\{1\leq m\leq m_{i, j}^{l, k};\ \lambda_{i, j}^{l, k}(m)(x)\in \textrm{B}(\partial V, \delta_k)\}<\eps,\quad\textrm{for any $x\in X_{j, k}$}.$$

Applying this for any $X_{j, k}$, there is $\delta>0$ such that $$\textrm{ocap}_{j, k}(\textrm{B}(\partial V, \delta))<\eps,\quad\textrm{for any $k$}$$ and hence $$\textrm{ocap}(\textrm{B}(\partial V, \delta))<{\eps}.$$ In particular, $$\mu_{\tau}(\textrm{B}(\partial V, \delta))<\eps$$ for any $\tau\in\textrm{T}(A)$.
\end{rem}

Using Remark \ref{tr-ocap} and a cardinality argument, one immediately has the following lemma.
\begin{lem}\label{baby-lisa}
{If $A$ has at most countably many extremal tracial states, then $A$ has the SBP.}
%Denote by $c$ the cardinality of the extremal tracial state of $A$. If $\aleph_0c$ is strictly less than the continuum, then $A$ has the SBP.
\end{lem}

\begin{defn}\label{SBRP}
An AH-algebra $A$ has small boundary refinement property (SBRP) if for any $X_{i, l}$, any finite open cover $\alpha$ of  $X_{i, l}$, and any $\eps>0$, there exists $L$ such that for any $X_{j, k}$ with $j>L$, there is a refinement $\alpha'\succ\alpha$ and $\delta>0$ (both may depend on $X_{j, k}$) such that there is a one-to-one correspondence between the elements $U$ of $\alpha$ and $U'$ of $\alpha'$ with $\overline{U'}\subseteq U$, and $$\mathrm{ocap}_{j, k}(\bigcup_{U'\in\alpha'}\textrm{B}(\partial U', \delta))<\eps,$$ where $\textrm{B}(\partial U', \delta)$ is the closed $\delta$-neighbourhood of $\partial U'$.
\end{defn}

\begin{lem}\label{judy}
If an AH-algebra has SBP, then it has SBRP.
\end{lem}
\begin{proof}
Using the small boundary property, there is a cover of $X_{i, l}$ by open sets with small boundary that refines $\alpha$. Then, by taking union of these sets, there is a refinement of $\alpha$ satisfying Definition \ref{SBRP}.
\end{proof}

\begin{thm}\label{nancy}
If a simple AH-algebra $A$ has SBRP, then $A$ has mean dimension zero.
\end{thm}
\begin{proof}
The proof is a modification of the proof of Theorem 5.4 of \cite{Lindenstrauss-Weiss-MD}. Assume that $A$ has the SBRP, and let $\alpha$ be an open cover of $X_i:=\bigsqcup_{l}X_{i, l}$. Fix $i$ and $l$, and consider the unit $e$ of $A_{i,l}$. Since $A$ is simple, there exists $\delta_l>0$ such that $\tau(e)>\delta_l$ for any tracial state $\tau$ on $A$. Therefore, there exists $N'$ such that for any $j>N'$, one has that $$\frac{n_{i, l}m_{i, j}^{l, k}}{n_{j, k}}>\delta_l,\quad\forall 1\leq k\leq h_j.$$

Let $d$ be a compatible metric on $X_{i, l}$.  For any $\eps>0$, since $A$ has SBRP, there exists $N>N'$ such that for any $X_{j, k}$ with $j>N$, there is exists an open cover $\alpha'\succ \alpha$ and $\delta>0$ such that there is a one-to-one correspondence between the elements $U_s$ of $\alpha_{i, l}$ and $U'_s$ of $\alpha'_{i, l}$ with $\overline{U_s'}\subset U_s$, such that $$\frac{n_{i, l}}{n_{j, k}}\#\{1\leq m\leq m_{i, j}^{l, k};\ \lambda_{i, j}^{l, k}(m)(x)\in \textrm{B}(\bigcup_{s=1}^{\abs{\alpha_{i, l}}}\partial U_s', \delta) \}<\eps\delta_l,\quad \forall x\in X_{j, k},$$ and $\textrm{B}(\partial U'_s, \delta)\subseteq U_s,$ where $\textrm{B}(\partial U'_s, \delta)$ is the closed $\delta$-neighbourhood of $\partial U'_s$.

%Therefore, there exists $N>N'$ such that for any $j>N$, one has $$\frac{n_{i, l}}{n_{j, k}}\#\{1\leq m\leq m_{i, j}^{l, k};\ \lambda_{i, j}^{l, k}(m)(x)\in \bigcup_{s=1}^{\abs{\alpha_{i, l}}}\partial U_s' \}<\eps\delta_l,\quad \forall x\in X_{j, k}.$$

Fix $X_{j, k}$ and let us show that $\mathcal{D}(\varphi_{i, j}^k(\alpha))/n_{j, k}$ is bounded by $\eps\abs{\alpha}$. %Choose $\delta>0$ such that $$\frac{n_{i, l}}{n_{j, k}}\#\{1\leq m\leq m_{i, j}^{l, k};\ \lambda_{i, j}^{l, k}(m)(x)\in B(\bigcup_{s=1}^{\abs{\alpha_{i, l}}}\partial U_s', \delta) \}<\eps\delta_l,\quad \forall x\in X_{j, k},$$ and $B(\partial U'_s, \delta)\subseteq U_s,$ where $B(\partial U'_s, \delta)$ is the closed $\delta$-neighbourhood of $\partial U'_s$.

Define
$$\phi'_s(x)=\left\{
\begin{array}{ll}
1, & \textrm{if}\ x\in U_s',\\
\max(0, 1-\delta^{-1}d(x, \partial U_s')), & \textrm{otherwise},
\end{array}
\right.
$$
and define
\begin{eqnarray*}
\phi_1(x)&=&\phi'_1(x),\\
\phi_2(x)&=&\min(\phi'_2(x), 1-\phi'_1(x)),\\
\phi_3(x)&=&\min(\phi'_3(x), 1-\phi'_1(x)-\phi'_2(x)),\\
&\vdots&.
\end{eqnarray*}
Therefore, one has a subordinate partition of unity $\phi_s: X_{i, l}\to [0, 1]$ such that 
\begin{enumerate}
\item $\displaystyle{\sum_{s=1}^{\abs{\alpha_{i, l}}}\phi_s(x)=1}$,
\item $\displaystyle{\mathrm{supp}(\phi_s)\subset U}$ for some $U\in\alpha,$
\item $\displaystyle{\frac{n_{i, l}}{n_{j, k}}\#\{1\leq m\leq m_{i, j}^{l, k};\ \lambda_{i, j}^{l, k}(m)(x)\in \bigcup_{s=1}^{\abs{\alpha_{i, l}}}\phi_s^{-1}(0, 1) \}<\eps\delta_l,\quad \forall x\in X_{j, k},\ \forall k.}$
%
%$\displaystyle{\mathrm{ocap}(\bigcup_{j=1}^{\abs{\alpha}}\phi_j^{-1}(0, 1))\leq\eps}$.
\end{enumerate}

Note that $j>N>N'$, one has that $$\frac{n_{i, l}m_{i, j}^{l, k}}{n_{j, k}}>\delta_l,\quad\forall 1\leq k\leq h_j.$$ Also note that $n_{j, k}=\sum_{l=1}^{h_i} n_{i, l}m_{i, j}^{l, k}$. Hence $$\sum_{l=1}^{h_i} m_{i, j}^{l, k}\leq n_{j, k}.$$

Set $E=\bigcup_{s=1}^{\abs{\alpha_{i, l}}}\phi_s^{-1}(0, 1))$, and choose $j$ large enough such that for any $1\leq l\leq h_i$, $$\frac{n_{i, l}}{n_{j, k}}\#\{1\leq m\leq m_{i, j}^{l, k};\ \lambda_{i, j}^{l,k}(m)(x)\in E\cap X_{i, l}\}<\eps\delta_l,\quad \forall x\in X_{j, k},\ \forall1\leq k\leq h_j.$$ Thus, 
\begin{equation}\label{007}
\frac{1}{m_{j, k}^{l, k}}\#\{1\leq m\leq m_{i, j}^{l, k};\ \lambda_{i, j}^{l,k}(m)(x)\in E\cap X_{i, l}\}<\eps,\quad\forall x\in X_{j, k},\ \forall1\leq k\leq h_j.
\end{equation}

Define $\Phi: X_i\to \Real^{\abs{\alpha}}$ by $$x\mapsto (\phi_1(x), ..., \phi_{\abs{\alpha}}(x)),$$ and set $$\Gamma_{i, j}^k: X_{j, k}\ni x\mapsto (\lambda_{i, j}^{l, k}(1)(x),..., \lambda_{i, j}^{l, k}(m_{i, j}^{l, k})(x))\in (X_i)^{m_{i, j}^{l, k}}.$$ 

Consider the map $$f_{i, j}^{l, k}={\Phi}\circ\Gamma_{i, j}^k: X_{j, k}\to \Real^{\abs{\alpha}m_{i, j}^{l, k}}.$$ Then the set $f_{i, j}^{l, k}(X_{j, k})$ is a subset of the image of finite number of $\eps\abs{\alpha}m_{i, j}^{l, k}$ dimensional affine subspace of $\Real^{\abs{\alpha}m_{i, j}^{l, k}}$. Indeed, let $e_j^i$, $i=1,..., m_{i, j}^{l,k}$, $j=1,...,\abs{\alpha}$ be the standard base of $\Real^{\abs{\alpha}m_{i, j}^{l,k}}$. For every $I=\{i_1,...,i_{N'}\}$, $N'\leq\eps m_{i, j}^{l, k}$, and every $\xi\in\{0, 1\}^{\abs{\alpha}m_{i, j}^{l, k}}$, set $$C(I, \xi)=\mathrm{span}\{e_j^i:\ i\in I, 1\leq j\leq\abs{\alpha}\}+\xi.$$ Then, by \eqref{007}, $$f_{i, j}^{l,k}(X_{j, k})\subseteq\bigcup_{\abs{I}<\eps m_{i, j}^{l, k}, \xi} \pi(C(I, \xi)).$$

It is easy to see that $f_{i, j}^{l, k}$ is $\varphi_{i, j}^{l, k}(\alpha_{i, l})$ compatible. By Proposition \ref{comp-map}, one has that $$\mathcal D(\varphi_{i, j}^{l, k}(\alpha_{i, l}))<\eps\abs{\alpha}m_{i, j}^{l, k},$$ and hence 
 \begin{eqnarray*}
 \mathcal{D}(\varphi_{i, j}^k(\alpha))&=&\mathcal{D}(\varphi_{i, j}^{1,k}(\alpha_1)\vee\cdots\vee \varphi_{i, j}^{h_i,k}(\alpha_{h_i}))\\
 &\leq&\mathcal{D}(\varphi_{i, j}^{1,k}(\alpha_1))+\cdots+\mathcal{D}(\varphi_{i, j}^{h_i,k}(\alpha_{h_i}))\\
 &\leq&\eps\abs{\alpha}(m_{i, j}^{1, k}+\cdots+m_{i, j}^{h_i, k})\\
 &\leq&\eps\abs{\alpha} n_{j, k}.
 \end{eqnarray*}

Therefore, $A$ has mean dimension zero, as desired.
\end{proof}

\begin{cor}\label{lisa}
If $A$ has at most countably many extremal tracial states, then $A$ has mean dimension zero.
\end{cor}
\begin{proof}
It follows from Lemma \ref{baby-lisa}, Lemma \ref{judy}, and Theorem \ref{nancy}.
\end{proof}

\begin{rem}\label{MD0>SBP}
Unlike the small boundary property for dynamical systems, which is equivalent to the mean dimension zero for minimal dynamical systems (Theorem 6.2 of \cite{Lind-MD}), the small boundary property for AH-algebra is strictly stronger than {the property of} mean dimension zero. The following is an example of a simple AH-algebra which has mean dimension zero but does not satisfy the small boundary property (it is one of the Goodearl algebras):

Let $(m_n)$ be a sequence of natural numbers such that there is $c>0$ such that $$(\frac{m_n-1}{m_n})(\frac{m_{n-1}-1}{m_{n-1}})\cdots(\frac{m_1-1}{m_1})>c,\quad\forall n\in\mathbb N.$$ Set $r_n=m_{n-1}\cdots m_1$ (assume $r_1=1$), and set $A_n=\textrm{M}_{r_n}(\textrm{C}([0, 1]))$. Let $\{x_n\}$ be a dense sequence in $[0, 1]$.  Set the eigenvalue maps (there are $m_n$ of them) between $A_n$ and $A_{n+1}$ to be $$\lambda_1=\textrm{id}, ...,\lambda_{m_n-1}=\textrm{id}, \lambda_{m_n}(x)=x_n.$$ Then the inductive limit $A=\varinjlim A_n$ is a simple AI-algebra, and hence has mean dimension zero (see Remark \ref{AH-SD-MD0}). However, it does not have SBP. In fact, the following stronger statement holds: for any $z\in[0, 1]$ in the spectrum of $A_n$, one has that $\mathrm{ocap}(\{z\})>0$.

Indeed, for any $k>n$, {there are $(m_n-1)(m_{n+1}-1)\cdots (m_{k-1}-1)$ of eigenvalue maps (in total of $m_nm_{n+1}\cdots m_{k-1}$) between $A_n$ and $A_k$ which are identity maps}. Thus, 
\begin{eqnarray*}
&&\sup_{x\in X_{j, k}}\frac{r_n}{r_k}\#\{1\leq m\leq m_nm_{n+1}\cdots m_{k-1};\ \lambda_m(x)=z\}\\
&\geq& (\frac{m_k-1}{m_k})(\frac{m_{k-1}-1}{m_{k-1}})\cdots(\frac{m_n-1}{m_n})\\
&\geq& (\frac{m_k-1}{m_k})(\frac{m_{k-1}-1}{m_{k-1}})\cdots(\frac{m_1-1}{m_1})>c,
\end{eqnarray*}
and $\textrm{ocap}(\{z\})>c$.

\end{rem}

\begin{defn}\label{VT-SVT}
{Let $F$ be a self-adjoint element in a homogeneous C*-algebra $$A_i=\bigoplus_{l=1}^h p_{i, l}\mathrm{M}_{r_{i, l}}(\textrm{C}(X_{i, l})) p_{i, l},$$contradicts 
where $X_{i, l}$ are connected compact Hausdorff spaces and $p_{i, l}$ are projections with ranks $n_{i, l}$}, the variation of the trace of $F$ is $$\max_l\sup_{s, t\in X_{i, l}}\frac{1}{n_{i, l}}\abs{\mathrm{Tr}(F(s))-\mathrm{Tr}(F(t))}.$$

An AH-algebra {$A=\varinjlim A_i$} is said to have the property of small variation of trace (SVT) if for any self-adjoint element $F\in A_i$ and any $\eps>0$, there exists $N>0$ such that for any $j>N$, the variation of the trace of $\varphi_{i, j}(F)$ is less than $\eps$.
\end{defn}

The following lemma is well known.

\begin{lem}\label{SVT}
If the projections of an AH-algebra separate traces, then it has the SVT.
\end{lem}
\begin{proof}
Let us show that if an AH-algebra $A$ does not have the SVT, then the projections of $A$ do not separate traces. Since $A$ does not have the SVT, there exist $\eps>0$, $F\in A_i$ and $j_m\to\infty$ as $m\to\infty$ such that for any $A_{j_m}$, there is a connected component $X_{j_m, k_m}$ and $s_{j_m}, t_{j_m}\in X_{j_m, k_m}$ such that 
$$\frac{1}{n_{j_m, k_m}}\abs{\mathrm{Tr}(\varphi_{i, j_m}(F)(s_{j_m}))-\mathrm{Tr}(\varphi_{i, j_{m}}(F)(t_{j_m}))}\geq\eps.$$ 
Therefore there are tracial states $\tau_{j_m}^{(0)}$ and $\tau_{j_m}^{(1)}$ on $A_{j_m}$ such that 
$$\abs{\tau_{j_m}^{(0)}(\varphi_{i, {j_m}}(F))-\tau_{j_m}^{(1)}(\varphi_{i, j_m}(F))}\geq\eps.$$ 
Pick $\omega\in\beta(\mathbb N)\setminus\mathbb N$, and consider 
$$x\mapsto\omega(\tau_{j_1}^{(l)}(\varphi_{i, j_1}(x)), \tau_{j_2}^{(l)}(\varphi_{i, j_2}(x)),...,\tau_{j_m}^{(l)}(\varphi_{i, j_m}(x)),...) $$ 
with $l=0, 1$. It then induces two tracial states $\tau^{(0)}$ and $\tau^{(1)}$ on $A$, and $\tau^{(0)}\neq\tau^{(1)}$. However, for any projection $p\in A$, one has that $\tau^{(0)}(p)=\tau^{(1)}(p)$. Thus the projections of $A$ do not separate traces.
\end{proof}

%\begin{lem}\label{flat}
%If the projections of an AH-algebra separate traces, then it has the following strong version of SVT: For any closed subset $D\subseteq X_i$ and any $\eps$, there exists $j>i$ such that $$\max_k\sup_{s, t\in X_{j, k}}\frac{1}{n_{j, k}}\abs{\mathrm{Tr}(1_D(s))-\mathrm{Tr}(1_D(t))}<\eps.$$
%\end{lem}
%\begin{proof}
%The argument is the same as above. If this were not true, then there exists $\eps>0$ and a closed subset $D\in A_i$ such that for any $A_j$ with $j>i$, there is a connected component $X_{j, k}$ and $s_j, t_j\in X_{j, k}$ such that $$\frac{1}{n_{j, k}}\abs{\mathrm{Tr}(\varphi_{i, j}(1_D)(s_j))-\mathrm{Tr}(\varphi_{i, j}(1_D)(t_j))}\geq\eps.$$ Therefore there are tracial states $\tau_j^{(0)}$ and $\tau_j^{(1)}$ on $A_j$ such that $$\abs{\tau_j^{(0)}(\varphi_{i, j}(1_D))-\tau_j^{(1)}(\varphi_{i, j}(1_D))}\geq\eps.$$ Pick $\omega\in\beta(\mathbb N)\setminus\mathbb N$, and consider $$x\mapsto\omega(\tau_i^{(l)}(x), \tau_{i+1}^{(l)}(\varphi_{i, i+1}(x)),...,\tau_{j}^{(l)}(\varphi_{i, j}(x)),...) $$ with $l=0, 1$. It then induces two tracial states $\tau^{(0)}$ and $\tau^{(1)}$ on $A$, and $\tau^{(0)}\neq\tau^{(1)}$. However, for any projection $p\in A$, one has that $\tau^{(0)}(p)=\tau^{(1)}(p)$, and this contradicts  the assumption.
%\end{proof}

\begin{defn}
Let $A$ be a C*-algebra, let $\tau_1$ and $\tau_2$ be two tracial states, and let $\F$ be a finite subset of $A$. We write $$\norm{\tau_1-\tau_2}_\F=\max\{\abs{\tau_1(f)-\tau_2(f)};\ f\in\F\}.$$

Let $\Delta$ be a subset of $\textrm{T}(A)$. Write $\tau\in_{\F, \eps}\Delta$ if there is $\tau'\in\Delta$ such that $\norm{\tau-\tau'}_\F<\eps.$

If $A$ is a homogeneous C*-algebra with spectrum $X$, for any $x\in X$, denote by $\tau_x$ the tracial state of $A$ induced by the Dirac measure on $\{x\}$.
\end{defn}

The following lemma can be regarded as a generalization of Lemma \ref{SVT}.
\begin{lem}\label{flat}
Let $A$ be a simple AH-algebra, and let $M>0$. If $\rho^{-1}(\kappa)$ has at most $M$ extreme points for any $\kappa\in  \mathrm{S}_u(\Kzero(A))$, then, for any $\eps>0$ and any finite subset $\F\subseteq A_i$, there exists $N$ such that for any $X_{j, k}$ with $j>N$, there exists a convex $\Delta_{j, k}\subseteq \mathrm{T}(A_i)$ with at most $M$ extreme points such that $$\varphi_{i, j}^*(\tau_{x})\in_{\F, \eps}\Delta_{j, k},\quad\forall x\in X_{j, k}.$$
\end{lem}
\begin{proof}
If this were not true, there exist a finite subset $\F\subseteq A_i$ and $\eps>0$ such that for any $n$, there exist $j_n>n$ and $k_n$ such that for any convex $\Delta\subseteq\textrm{T}(A_i)$ with at most $M$ extreme points, there exists $x_{j_n, k_n}(\Delta)$ such that 
\begin{equation}\label{faraway}
\varphi_{i, j_n}^*(\tau_{x_{j_n, k_n}(\Delta)})\notin_{\F, \eps}\Delta
\end{equation}
Let us show that it is impossible.

Choose an arbitrary point $y_{j_n, k_n}\in X_{j_n, k_n}$, and consider the tracial state $\tau_{y_{j_n, k_n}}$. Extend it to a state of $A$, and consider an accumulation point $\tau_y$. It is a straightforward argument to show that $\tau_y$ is in fact a tracial state of $A$. 

Consider $$\Delta_\infty:=\rho^{-1}(\rho(\tau_y)).$$ It then has at most $M$ extreme points, hence $$\Delta_i:=\varphi_{i, \infty}^*(\Delta_\infty)$$ has at most $M$ extreme points.

Consider the points $x_{j_n, k_n}(\Delta_i)$ and the tracial states $\tau_{x_{j_n, k_n}(\Delta_i)}$. The same argument as above shows that there is a tracial state $\tau_x$ on $A$ such that $\tau_x$ is an accumulation point of $\{\tau_{x_{j_n, k_n}(\Delta_i)}\}$. However, since $x_{j_n, k_n}$ is in the same connected component {containing} $y_{j_n, k_n}$, one has $$\rho(\tau_x)=\rho(\tau_y)$$ and hence $$\tau_x\in\Delta_\infty.$$ Since $\tau_x$ is an accumulation point of $\{\tau_{x_{j_n, k_n}(\Delta_i)}\}$, one has that $$\norm{\varphi^*_{j_n, \infty}(\tau_x)-\tau_{x_{j_n, k_n}(\Delta_i)}}_\F<\eps$$ for sufficiently large $n$, thus $$\norm{\varphi^*_{i, \infty}(\tau_x)-\varphi^*_{i, j_n}\tau_{x_{j_n, k_n}(\Delta_i)}}_\F<\eps,$$ which {contradicts} \eqref{faraway}, as desired.
\end{proof}

\begin{thm}\label{RR0-SBP}
{Let $A$ be a simple AH-algebra such that all eigenvalue patterns are induced by continuous functions}. If there is $M>0$ such that $\rho^{-1}(\kappa)$ has at most $M$ extreme points for any $\kappa\in\mathrm{S}(\Kzero(A))$, then $A$ has SBRP, and hence it has mean dimension zero. In particular,  {such} AH-algebras with real rank zero have mean dimension zero.
\end{thm}

\begin{proof}
Consider $X_{i, l}$, and let $d$ be a compatible metric on $X_{i, l}$. We assert that in order to prove the theorem, it is enough to show that for any open ball $\mathrm{B}(x, t)$, any $\eps>0$, and any $s<t$, there exists $L>0$ such that for any $X_{j, k}$ with $j>L$, there are $s<r_1<r_2<t$ such that $$\textrm{ocap}_{j, k}(\mathrm{Ann}(x, r_1, r_2))<\eps,$$ where $\mathrm{Ann}(x, r_1, r_2)$ is the closed annulus with centre $x$ and radius $r_1$ and $r_2$.

Indeed, let $\alpha$ be a finite open cover of $X_{i, l}$. and denote by $\delta$ the Lebesgue number of $\alpha$. Let $\{\textrm{B}(x_n, s_n);\ n=1, ..., m\}$ a finite cover of $X_{i, l}$ with open balls such that for each $\textrm{B}(x_n, s_n)$, one has that $\overline{\textrm{B}(x_n, s_n)}\subseteq U$ for some $U\in\alpha$.

For each $\textrm{B}(x_n, s_n)$, choose $t_n>s_n$ such that $\textrm{B}(x_n, t_n)\subset U$ if $\overline{\textrm{B}(x_n, s_n)}\subseteq U$. Choose $L_n>0$ such that for any $X_{j, k}$ with $j>L_n$, there are $s_k<r_{k, 1}<r_{k, 2}<t_n$ such that 
$$\textrm{ocap}_{j, k}(\mathrm{Ann}(x_n, r_{n, 1}, r_{n, 2}))<\eps/m.$$ 
Set $L=\max\{L_1, ..., L_m\}$, and consider a space $X_{j, k}$ with $j>L$. 

For each $U\in\alpha$, set 
$$U'=\bigcup_{\textrm{B}(x_n, s_n)\subset U}\textrm{B}(x_n, r_n),$$ where $r_n=(r_{n, 1}+r_{n, 2})/2$. Set $\delta=\min\{(r_{n, 2}-r_{n, 1})/2;\ 1\leq n\leq m\}$.

Then 
$$\bigcup_k U_k'\supseteq\bigcup_{n=1}^m\textrm{B}(x_n, r_n)\supseteq\bigcup_{n=1}^m\textrm{B}(x_n, s_n)=X_{i, l},$$ 
and hence $\{U_k'\}$ is an open cover of $X_{i, l}$. Moreover, one has that $U'_k\subset U_k$, and 
$$\textrm{B}(\partial U'_k, \delta)\subseteq\bigcup_{n=1}^m\textrm{Ann}(x_n, r_{n, 1}, r_{n, 2})).$$ 
Hence 
$$\mathrm{ocap}_{j, k}(\bigcup_{j=1}^{\abs{\alpha}}\partial U_j')\leq\sum_{j=1}^m\mathrm{ocap}_{j, k}(\textrm{Ann}(x_j, r_{j, 1}, r_{j, 2}))<\eps.$$ 
This proves the assertion.

Consider a ball $\textrm{B}(x, t)$, any $\eps>0$ and $s<t$. Set $m=[2/\eps]^M$, where $[1/\eps]$ is the smallest integer bigger than $1/\eps$. Set $d=t-s$. For each $1\leq n\leq m$, choose a continuous function $0\leq f_n\leq 1$ such that $f_n(y)=0$ for any $y\notin\textrm{Ann}(x, s+(n-1)d/m, s+nd/m)$ and $f_n(y)=1$ for any $y\in\textrm{Ann}(x, s+(4n-3)d/4m, s+(4n-1)d/4m)$. 

It is clear that $f_n$ are mutually orthogonal. Denote by $\F=\{f_1, ..., f_m\}$. By Lemma \ref{flat}, there exists $L$ such that for any $X_{j, k}$ with $j>L$, there exists a convex $\Delta_{j, k}\subseteq\textrm{T}(A_i)$ with at most $M$ extreme points such that  $$\varphi^*_{i, j}(\tau_x)\in_{\F, \eps/2}\Delta_{j, k}, \quad\forall x\in X_{j, k}.$$

Denote by $\mu_1, ..., \mu_M$ the extreme points of $\Delta_{j, k}$. Then, there exists $f_n\in\F$ such that $$\mu_t(f_n)<\eps/2,\quad 1\leq t \leq M,$$ and hence $$\varphi^*_{i, j}(\tau_x)(f_n)<\eps/2+\eps/2=\eps, \quad\forall x\in X_{j, k}.$$

Set $r_1=s+(4n-3)d/4m$ and $r_2=s+(4n-1)d/4m$.  Note that $$\textrm{ocap}_{j, k}(\textrm{Ann}(x, r_1, r_2))\leq \max_{x\in X_{i, j}}\{\varphi^*_{i, j}(\tau_x)(f_n)\}<\eps,$$ as desired.
\end{proof}

\begin{rem}
The author would like to thank Professor XiaoQiang Zhao for encouraging him to fill a gap in the proof of the theorem above.
\end{rem}

\begin{rem}
In the proof Theorem \ref{RR0-SBP}, in general, one cannot choose the annulus $\textrm{Ann}(x, r_1, r_2)$ uniformly small with respect to $k$, i.e., $r_1$ and $r_2$ cannot be chosen independent of $k$. Indeed, if $\textrm{Ann}(x, r_1, r_2)$ were uniformly small with respect to $k$, then, one has that $\tau(f)<\eps$ for any $\tau\in \textrm{T}(A)$ and any centre element $f$ in $A_i$ with $\textrm{supp}(f)\subseteq \textrm{Ann}(x, r_1, r_2)$ and $0\leq f\leq 1$. The following is an example of AH-algebra $A$ such that this statement does not hold. 

In fact, we shall construct an AH-algebra $A=\varinjlim(A_n, \varphi_{n, n+1})$ with real rank zero, and find $\delta>0$ such that for any $A_n$ and any subset $D$ in the spectrum of $A_n$ with nonempty interior, there exists $\tau\in\textrm{T}(A)$ with $\mu_\tau(D)>\delta$. In particular, $A$ does not have (SBP) (see Remark \ref{rem-sbp}), but has (SRBP).

Let $m_1, ..., m_n,...$ be a sequence of positive integers such that 
$$\frac{m_1}{m_1+1}\cdot\frac{m_2}{m_2+2} \cdots \frac{m_n}{m_n+2^{n-1}}\geq\frac{1}{2}\quad\textrm{for any $n$}.$$
Set $c_n=(m_1+1)(m_2+2)\cdots(m_{n-1}+2^{n-1})$ (assume $r_1=1$). Consider $A_n=\bigoplus_{2^{n-1}}\textrm{M}_{c_n}(\textrm{C}([0, 1]))$, and define the map $\varphi_{n, n+1}$ as the following: For any $1\leq k\leq 2^{n}$, set $i=\lfloor k/2\rfloor\mod 2$. The eigenvalue maps for the restriction of $\varphi_{n, n+1}$ to $A_n$ and $k^{\textrm{th}}$ copy of $A_{n+1}$ are 
\begin{eqnarray*}
\lambda_1(x)=\frac{x+i}{2}&\in& \textrm{$(\lfloor \frac{k}{2}\rfloor+1)^{\textrm{th}}$ copy of $[0, 1]$},\\
&\vdots&\\
\lambda_{m_n}(x)=\frac{x+i}{2}&\in& \textrm{$(\lfloor \frac{k}{2}\rfloor+1)^{\textrm{th}}$ copy of $[0, 1]$},\\
\lambda_{m_n+1}(x) = \frac{1}{2}&\in& \textrm{first copy of $[0, 1]$},\\
&\vdots&\\
\lambda_{m_n+2^{n-1}}(x) = \frac{1}{2}&\in& \textrm{$(2^{n-1})^{\textrm{th}}$ copy of $[0, 1]$}.
\end{eqnarray*}

Then the inductive limit $A=\varinjlim(A_n, \varphi_{n, n+1})$ is an AI-algebra. A routine calculation shows that $A$ is simple. Moreover, an approximate intertwining argument shows that $A$ is in fact an AF-algebra.  Hence $A$ has real rank zero, and satisfies the condition of Theorem \ref{RR0-SBP}. 

One asserts the following: Consider any $A_n$ and any copy of $[0, 1]$ in its spectrum. For any nonzero $l\in\mathbb{N}$ and any $0\leq j\leq 2^l-1$, there exists $\tau\in\textrm{T}(A)$ such that $$\mu_\tau([\frac{j}{2^l}, \frac{j+1}{2^l}])>\frac{1}{2}.$$ In particular, any subset with nonempty interior---for example $\textrm{Ann}(x, r_1, r_2)$---cannot be uniformly small. Indeed, for any $d>l$, there is a copy of $[0, 1]$ in the spectrum of $A_{n+d}$ such that the eigenvalue maps of the restriction $\varphi_{n, n+d}$ to $A_n$ and the corresponding direct summand of $A_{n+d}$ are 
\begin{eqnarray*}
\lambda_1(x)=\frac{x+j}{2^d}&\in& [\frac{j}{2^l}, \frac{j+1}{2^l}],\quad x\in[0, 1],\\
&\vdots&\\
\lambda_{L}(x)=\frac{x+j}{2^d}&\in& [\frac{j}{2^l}, \frac{j+1}{2^l}], \quad x\in[0, 1],\\
\lambda_{L+1}(x) &=& \textrm{constant},\\
&\vdots&\\
\lambda_{C}(x) &=& \textrm{constant},
\end{eqnarray*} 
where ${L}=m_n\cdots m_{n+d-1}$ and $C=(m_n+2^{n})\cdots(m_{n+d-1}+2^{n+d-1})$. Then there exists $x_d\in [0, 1]$ such that $$\mu_{\tau_{x_d}}([\frac{j}{2^l}, \frac{j+1}{2^l}])\geq\frac{L}{C}\geq\frac{1}{2}.$$

Extend $\tau_{x_d}$ to a state of $A$ and consider an accumulation point $\tau$ of $\{\tau_{x_d}; d=l, l+1, ...\}$. Then $\tau$ is a tracial state of $A$ and $$\tau([\frac{j}{2^l}, \frac{j+1}{2^l}])\geq\frac{1}{2}.$$ This proves the assertion.
\end{rem}

\section{A local approximation theorem}

Let $X$ be a compact Hausdorff space. Since $\mathcal{D}(\alpha)\leq\mathrm{dim}(X)$ for any open cover $\alpha$ of $X$, one has that if an AH-algebra $A$ has slow dimension growth, then $A$ has mean dimension zero. 

In this section, let us consider AH-algebras with diagonal maps. Recall that a connecting map $\varphi_{i, j}: A_i\to A_j$ is an diagonal map if there exists a unitary $U\in A_j$ such that the restriction of $\mathrm{ad}(U)\circ\varphi_{i, j}$ to $A_{i, l}$ and $A_{j, k}$ has the form $$f\mapsto\left(
\begin{array}{ccc}
f\circ\lambda_1 & &\\
&\ddots&\\
 && f\circ\lambda_n
\end{array}
\right)
$$
for some continuous maps $\lambda_1,...,\lambda_n : X_{j, k}\to X_{i, l}$. 

{
\begin{rem}
It is interesting to point out that in \cite{EHT-sr1} the authors showed that any simple AH-algebra with diagonal maps has stable rank one, i.e., invertible elements are dense. This generalized the conclusion of Proposition 10 of \cite{Vill-perf} from the first type of Villadsen algebras to arbitrary simple AH-algebras with diagonal maps. So, is it possible that stable rank one exactly characterizes simple AH-algebras with diagonal maps? That is, given an arbitrary simple AH-algebra $A$, if $A$ has stable rank one, does $A$ has to be an AH-algebra with diagonal maps?
\end{rem}
}

\begin{rem}
If $A$ is an AH-algebra with diagonal maps, then one can assume that all the connecting maps has the form above. 

For any $f\in A_i$, if there are open sets $U_1,...,U_n$ such that $\norm{f(x)-f(y)}\leq\eps$ for any $x, y\in U_i$ $1\leq i\leq n$, then, for any $\displaystyle{x, y\in\bigcap\lambda^{-1}(U_i)}$, one has that $\norm{\varphi(f)(x)-\varphi(f)(y)}\leq\eps.$ Therefore, for any $f\in A_i$, any $\eps>0$, and any open cover $\alpha$ of $X_i$ satisfying $\norm{f(x)-f(y)}\leq\eps$, $\forall x, y\in U$, $\forall U\in\alpha$, one has 
$$\norm{\varphi_{i, j}(f)(x)-\varphi_{i, j}(f)(y)}\leq\eps,\quad\forall x, y\in V,\ \forall V\in\varphi_{i, j}(\alpha).$$
\end{rem}

\begin{thm}\label{LOA}
Let $A$ be an AH-algebra with diagonal map, and denote by $\gamma$ the mean dimension of $A$. Then for any finite subset $\F\subseteq A$ and any $\eps_1>0$ and $\eps_2>0$, there exists a unital sub-C*-algebra $$C\cong\bigoplus p_i\MC{r_i}{\Omega_i}p_i\subseteq A$$ such that $\F\subseteq_{\eps_1} C$, and $$\frac{\mathrm{dim}{\Omega_i}}{\mathrm{rank}(p_i)}<\gamma+\eps_2.$$
\end{thm}
\begin{proof}
{Write $A=\varinjlim A_i$ with $$A_i=\bigoplus_{l=1}^{h_i} p_{i, l} \textrm{M}_{r_{i, l}}(\textrm{C}(X_{i, l})) p_{i, l}.$$ Denote by $X_i$ the disjoint union of $X_{i, 1}, ..., X_{i, h_i}$, and denote by $n_{i, l}$ the rank of $p_{i, l}$.}

Without loss of generality, one may assume that {$1_A\in \F\subseteq A_1$ and each element of $\mathcal F$ has norm at most one, and $\eps_1<1$}. Choose an open cover $\alpha$ of $X_1$ such that  for any $U\in \alpha$, any $x, y\in U$, and any $f\in \F$, one has $$\norm{f(x)-f(y)}<\frac{\eps_1}{20}.$$ Since $A$ has mean dimension $\gamma$, there exists $A_j$ such that $$\frac{\mathcal{D}(\alpha_{1, j}^k)}{n_{j, k}}<\gamma+\eps_2$$ for any $k$. Without loss of generality, assume that $j=2$.

On each $X_{2, k}$, consider an open cover $\beta$ with $\beta\succ\alpha_{1,2}^k$ with $$\mathrm{ord}(\beta)=\mathcal{D}(\alpha_{1,2}^k).$$

Note that for any $U\in \beta$, any $x, y\in U$, and any $f\in\F_1$,
\begin{equation*}
\norm{\varphi_{1, 2}^k(f)(x)-\varphi_{1, 2}^k(f)(y)}<\frac{\eps_1}{20}
\end{equation*}

Choose a partition of unity $\phi_i: X_2\to[0, 1]$ subordinate to $\beta$, and choose $x_i\in U_i$ for each $i$. Then, for any $x\in X_2$ and any $f\in\F$, one has
\begin{eqnarray*}
\norm{\sum_i\phi_i(x)\varphi_{1, 2}(f)(x_i)-\varphi_{1, 2}(f)(x)}&=&\norm{\sum_i\phi_i(x)\varphi_{1, 2}(f)(x_i)-\sum_{i}\phi_i(x)\varphi_{1, 2}(f)(x)}\\
&=&\norm{\sum_i\phi_i(x)(\varphi_{1, 2}(f)(x_i)-\varphi_{1, 2}(f)(x))}\\
&<&\frac{\eps_1}{20}.
\end{eqnarray*}

Define a c.p. map $\Theta:  \bigoplus_k\MC{r_{2, k}}{X_{2, k}} \to  \bigoplus_k\MC{r_{2, k}}{X_{2, k}}$ by $$f(x)\mapsto \sum_i\phi_i(x)f(x_i),$$ and therefore, one has 

\begin{equation*}
\norm{\Theta(f)-f}<\frac{\eps_1}{20},\quad\forall f\in\F.
\end{equation*}

Consider the $\mathrm{ord}(\beta)$-dimensional simplicial complex $\Delta$ constructed as follows: The vertices of $\Delta$ correspond to the elements of $\beta$, the $s$-dimensional simplices correspond to all $U_1, ..., U_s$ with ${\bigcap U_i\neq \textrm{\O}}$. The point in each simplex $\{U_1, ..., U_s\}$ can be parameterized as 
$$\sum_{i=1}^s\lambda_1[U_i]\quad\textrm{with $\lambda_i\geq 0$ and $\sum_i\lambda_i=1.$}$$
Define a map $\xi: X_2\to \Delta$ by $$x\mapsto \sum_{x\in U_i\in\beta}\phi_i(x)[U_i],$$ and it induces a *-homomorphism $$\Xi: \bigoplus_k\MC{r_{2, k}}{\Delta_{k}}\to {\bigoplus_k\MC{r_{2, k}}{X_{2, k}}}.$$

For each $f\in A_2$, define function $\breve{f}$ on $\Delta$ as follows: On vertices, set $$\breve{f}([U_i])={f}(x_i)$$ and extend it linearly to the simplicial complex. Then, it is clear that $$\Xi(\breve{f})=\Theta(f).$$ Therefore, the image of $\Theta$ is contained in the image of $\Xi$, and hence $\F$ can be approximated by elements in the image of $\Xi$ within $\frac{\eps_1}{20}$.

Identify the image of $\Xi$ with $$C'=\bigoplus_k\MC{r_{k}}{\Omega_{k}},$$ where $r_k=r_{2, k}$ and $\Omega_k$ the corresponding closed subset of $\Delta_{2, k}$. Since $1_{A}\in\mathcal F$, there is a projection $p'\in C'$ 
such that $$\norm{1_{A_2}-p'}<\frac{\eps_1}{10},$$
and hence
$$\norm{f-p'(\Xi(\breve{f}))p'}=\norm{1_{A_2}f1_{A_2}-p'(\Xi(\breve{f}))p'}<\frac{3}{20}\eps_1,\quad\forall f\in \mathcal F.$$
Therefore, any element of $\cal F$ in fact can be approximated by an element in $p'C'p'$ within $\frac{3}{20}\eps_1$. Since $\eps_1<1$, there is a unitary $u\in A_2$ such that 
$$u^*p'u=1_{A_2}\quad\textrm{with}\quad \norm{1-u}<\frac{\sqrt{2}}{10}\eps_1<\frac{\eps_1}{5}.$$

Consider the C*-algebra $$C:=u^*(p'C'p')u.$$ It is then a unital sub-C*-algebra of $A_2$, and each element of $\mathcal F$ can be approximated by an element of $C$ up to $\frac{3}{20}\eps_1+(1+\frac{3}{20}\eps_1)\frac{2}{5}\eps_1<\eps_1$.

Since $p'$ is unitarily equivalent to $1_{A_2}$, one has that $p'=\bigoplus_k p_k$ with $p_k$ a projection in $\textrm{M}_{r_k}(\textrm{C}(\Omega_k))$ with rank $n_{j, k}$. Therefore $C$ is isomorphic to $$\bigoplus p_k\MC{r_k}{\Omega_k}p_k.$$

Note that each $\Omega_{k}$ is a closed subset of $\Delta_{2, k}$, and hence for each $k$,
\begin{equation*}
\mathrm{dim}(\Omega_{k})\leq\mathrm{ord}(\beta)=\mathcal{D}(\alpha_{1,2}^k). 
\end{equation*}
Therefore, one has $$\frac{\mathrm{dim}(\Omega_k)}{\mathrm{rank}(p_k)}\leq \frac{\mathcal{D}(\alpha_{1,2}^k)}{n_k}<\gamma+\eps_2,$$ as desired.
\end{proof}

%\begin{rem}
If an AH-algebra with diagonal map has mean dimension zero, then, it has local slow dimension growth by Theorem \ref{LOA}. Then the following theorem states that the strict order on projections is in fact determined by their traces.
%\end{rem}

\begin{thm}\label{K0-comp}
Let $A$ be a {simple} AH-algebra with diagonal map. If $A$ has mean dimension zero, then, for any projections $p$ and $q$ in $A$, if $\tau(p)>\tau(q)$ for any tracial state $\tau$ of $A$, the projection $q$ is Murray-von Neumann equivalent to a sub-projection of $p$. 
\end{thm}
\begin{proof}
Since the simplex of tracial states is compact {and $A$ is simple}, there exists $\delta>0$ such that $\tau(p)>\tau(q)+\delta$ for any tracial state $\tau$ on $A$.

We assert that one can find a homogeneous C*-algebra $C$ satisfying Theorem \ref{LOA} such that $\{p, q\}\subseteq_{1/4} C$ and for any tracial state $\tau$ on $C$, one has $$\tau(p')>\tau(q')+\delta,$$ where $p'$ and $q'$ are projections in $C$ which are close to $p$ and $q$ within $1/2$ respectively.

If this were not true, there is a sequence of {unital sub-C*-algebras $\{C_n\}$ which satisfies the conclusion of Theorem \ref{LOA}} such that for any $n$, there is a tracial state $\tau_n$ on $C_n$ with $$\tau_n(p')\leq\tau_n(q')+\delta,$$ and moreover, {the union of $\{C_n\}$ is dense in $A$}. Extend each $\tau_n$ to a state of $A$, and still denote it by $\tau_n$.

Pick an accumulation point $\tau_\infty.$ It is clear that $\tau_\infty$ is a trace on $A$. However, $$\tau_\infty(p)=\lim\tau_n(p')\leq\lim\tau_n(q')+\delta=\tau_\infty(q)+\delta,$$ which contradicts  the assumption. This proves the assertion.

Therefore, by Theorem \ref{LOA}, there exists {$C\cong\bigoplus p_i\MC{r_i}{\Omega_i}p_i\subseteq A$ } with $$\frac{\mathrm{dim}{\Omega_i}}{n_i}<\delta,$$ {where $n_i=\mathrm{rank}(p_i)$}
such that $\{p, q\}\subseteq_{1/4} C$, and for any tracial state $\tau$ on $C$, 
\begin{equation}\label{cp-trace0}
\tau(p')>\tau(q')+\delta,
\end{equation} where $p'$ and $q'$ are projections in $C$ which are close to $p$ and $q$ within $1/2$ respectively. Note that the projections $p$ and $q$ are unitary equivalent to the projections $p'$ and $q'$ respectively. Therefore, in order to prove the theorem, it is enough to show that $q'$ is Murray- von Neumann equivalent to a sub-projection of $p'$.

Indeed, it follows from \eqref{cp-trace0} that $$\mathrm{rank}(p'(x))>\mathrm{rank}(q'(x))+\delta n_i\quad\textrm{if}\ x\in \Omega_i.$$ Since $\mathrm{dim}(\Omega_i)<n_i\delta$, by Theorem (9)1.2 and Theorem (9)1.5 of \cite{Husemoller-Fibre-Bundles}, one has that $q'$ is Murray-von Neumann equivalent to a sub-projection of $p'$, as desired.
\end{proof}

\begin{cor}
Any simple unital real rank zero AH-algebra with diagonal maps (without assuming slow dimension growth) has slow dimension growth.
\end{cor}
\begin{proof}
By Theorem \ref{RR0-SBP}, such an AH-algebra has mean dimension zero. It then follows from Theorem \ref{K0-comp} that the $\Kzero$-group has the strict comparison, and hence it is tracially AF by Theorem 2.1 of \cite{Lin-AHRR0} or Theorem 1.5 of \cite{Ng-LAH-RR0}. By the classification of tracially AF algebras, it has slow dimension growth.
\end{proof}

\section{Mean dimension zero and AH-algebras with diagonal maps}

In this section, the Cuntz semigroup of a {simple unital} AH-algebra $A$ with diagonal maps {and mean dimension zero is} calculated to be $\mathrm{V}(A)\sqcup\mathrm{SAff}(\mathrm{T}(A))$. Together with a recent result of W. Winter, this implies that $A$ is $\mathcal Z$-stable, and hence $A$ is an AH-algebra without dimension growth.

The following Lemma is due to M. R{\o}rdam. See 4.3 of \cite{Toms-PLMS} for a proof.
\begin{lem}\label{Rordam1}
Let $A$ be a C*-algebra, and let $\eps>0$. Let $b$ and $b'$ be positive elements with $\norm{b'-(b-\eps/2)_+}\leq\eps/4$. Set $\tilde{b}=(b'-\eps/4)_+$. Then one has that $\norm{b-\tilde{b}}<\eps$ and $$(b-\eps)_+\precsim\tilde{b}\precsim(b-\eps/2)_+\precsim b.$$
\end{lem}

\begin{defn}
An ordered abelian semigroup $(W, +, \leq)$ is said to be almost unperforated if for any element $a$ and $b$, if $(n+1)a\leq n b$ for some $n\in\mathbb{N}$, then $a\leq b$.
\end{defn}

One then has the following lemma.
\begin{lem}[Corollary 4.6 of \cite{Ror-Z-stable}]
Let $A$ be a simple unital exact C*-algebra. If $W(A)$ is almost unperforated, then $A$ has strict comparison of positive elements.
\end{lem}

The dimension function $d_\tau(\cdot)$ is lower semicontinuous, both with $\cdot$ and with $\tau$.
\begin{lem}\label{LSC-EL}
For any state $\tau$ on $A$, the map $a\to d_{\tau}(a)$ is lower semicontinuous, i.e., for any $\eps>0$, there exists $\delta>0$ such that for any $b$ with $\norm{a-b}\leq\delta$, one has $d_\tau(b)>d_\tau(a)-\eps$.
\end{lem}
\begin{proof}
There exists $\eps_1$ such that $$\abs{d_\tau(a)-\tau(f_{\eps_1}(a))}<\eps/2.$$ Choose $\delta>0$ such that $$\norm{f_{\eps_1}(a)-f_{\eps_1}(b)}<\eps/2$$ whenever $\norm{a-b}\leq\delta$. One then has that $$d_\tau(b)=\sup_{\eps'>0}\tau(f_{\eps'}(b))\geq \tau(f_{\eps_1}(b))>\tau(f_{\eps_1}(a))-\eps/2>d_{\tau}(a)-\eps,$$ as desired.
\end{proof}

\begin{lem}\label{LSC-TR}
For any positive element $a$ in $A$, the map $\tau\to d_{\tau}(a)$ is lower semicontinuous, i.e., for any $\eps>0$, there exists a neighbourhood $U\ni\tau$ such that for any $\rho\in U$, one has $d_\rho(a)>d_\tau(a)-\eps$.
\end{lem}
\begin{proof}
There exists $\eps_1$ such that $$\abs{d_\tau(a)-\tau(f_{\eps_1}(a))}<\eps/2.$$ Choose a neighbourhood $U$ in the simplex of traces of $A$ such that $$\abs{\rho(f_{\eps_1}(a))-\tau(f_{\eps_1}(a))}<\eps/2,\quad\forall \rho\in U.$$ One then has $$d_\rho(a)=\sup_{\eps'>0}\rho(f_{\eps'}(a))\geq \rho(f_{\eps_1}(a))>\tau(f_{\eps_1}(a))-\eps/2>d_{\tau}(a)-\eps,$$ as desired.
\end{proof}

\begin{lem}\label{localize}
{Let $A$ be a C*-algebra, and let $\{C_i\}$ be a collection of sub-C*-algebra of $A$ such that for any finite subset $\F\subseteq A$ and any $\eps>0$, there is $C_i$ with $\F\subseteq_{\eps} C_i$.}

Then, for any positive elements $a$ and $b$ in $A$ with $a\precsim b$, and any $\eps>0$, there exist $C_i$ and positive elements $a_i, b_i\in C_i$ such that $$\norm{a-a_i}\leq\eps,\quad\norm{b-b_i}\leq\eps, \quad\textrm{and}\quad (a_i-\eps)_+\precsim b_i$$ in $C_i$, and $b_i\precsim b$.
\end{lem}
\begin{proof}
Since $a\precsim b$, there is a sequence $\{v_i\}$ in $A$ such that $$\lim_{i\to\infty}v^*_ibv_i=a.$$ For each $i$, choose a sub-C*-algebra $C_i$ with $u_i, a_i, b_i$ in $C_i$ such that $$\norm{v_i^*bv_i-u^*_ib_iu_i}\leq 1/2^i\quad\textrm{and}\quad\norm{a_i-a}\leq 1/2^i.$$ Moreover, by Lemma \ref{Rordam1}, $b_i$ can be chosen so that $b_i\precsim b$. Hence one has $$\lim_{i\to\infty} u_i^*b_iu_i=a$$ and $$\lim_{i\to\infty} a_i=a\quad \textrm{and}\quad \lim_{i\to\infty} b_i=b.$$ Moreover, $b_i\precsim b.$

Choose $C_i$ large enough such that $\norm{a-a_i}<\eps/4$ and $\norm{b-b_i}<\eps/4$, and $$\norm{u_i^*b_iu_i-a}<\eps/4.$$ Hence, one has $$\norm{u_i^*b_iu_i-a_i}<\eps/2,$$ and therefore $$(a_i-\eps)_+\precsim b_i$$ in $C_i$, and $b_i\precsim b$, as desired.
\end{proof}

{Recall that
\begin{defn}[\cite{RC-Toms}]
A C*-algebra $A$ has radius of comparison less than $r$, write $\mathrm{rc}(A)<r$, if for any positive elements $a, b\in A$ with $$d_\tau(a)+r< d_\tau(b),\quad\forall \tau\in\mathrm{T}(A),$$ one has that $a\precsim b$. The radius of comparison of $A$ is $\mathrm{rc}(A):=\inf\{r;\ \mathrm{rc}(A)<r\}$.
\end{defn}
}

\begin{lem}\label{LOA-RC}
Let $A$ be a simple C*-algebra satisfying the following property: For any finite subset $\F\subseteq A$ and any $\eps_1, \eps_2>0$, there is a sub-C*-algebra $C$ such that $\F\subseteq_{\eps_1} C$ and $\mathrm{rc}(C)<\eps_2$, then $W(A)$ is almost unperforated. {In particular, if $A$ is also exact, it} has the strict comparison of positive elements.
\end{lem}
\begin{proof}
Assume that one has positive elements $a$ and $b$ in $A$ with  $(n+1)[a]\precsim n[b]$ for some $n$.

Fix $\eps>0$ for the time being. It follows from Lemma \ref{localize} that there exist $C_i$ and $a_i, b_i\in C_i$ such that $$\norm{a_i-a}\leq\eps,\quad \norm{b_i-b}\leq\eps\quad \textrm{and}\quad b_i\precsim b, $$ and 
\begin{equation}\label{eq310}
\bigoplus_{n+1}(a_i-\eps)_+\precsim \bigoplus_nb_i
\end{equation} 
in $C_i$. Moreover, using the simplicity of $A$, the compactness of the simplex of tracial states of $A$, and the lower-semicontinuity of $d_\tau$, one may assume that there is a strictly positive number $c$ which is independent of $C_i$ such that $$d_\tau(b_i)>c$$ for any tracial state $\tau$ on $C_i$. 

Indeed, for any $\eps'>0$, consider $f_{\eps'}(b_i)$. It is clear that $f_{\eps'}(b_i)\to f_{\eps'}(b)$. Since $A$ is simple, there exists $c>0$ such that $$\tau(f_{\eps'}(b))>c$$ for any tracial state $\tau$ on $A$.  Then there exists a sub-C*-algebra $C_i$ such that 
\begin{equation}\label{3031}
\tau(f_{\eps'}(b_i))>c
\end{equation} for all tracial state $\tau$ on $C_i$. If this were not true, there is a sequence of $C_i$ with dense union in $A$, positive elements $b_i\in C_i$ and a sequence of tracial state $\tau_i$ on $C_i$ such that $$\norm{f_{\eps'}(b_i)-f_{\eps'}(b)}\leq 1/2^i\quad\textrm{and}\quad{\tau_i}(f_{\eps'}(b_i))\leq c.$$ Extend $\tau_i$ to a state of $A$, and pick an accumulation point $\tau_\infty$, and assume that $\tau_i\to\tau_\infty$. One then has $$\tau_i(f_{\eps'}(b))\leq c+1/2^i,$$ and  $$\tau_\infty(f_{\eps'}(b))=\lim_{i\to\infty} \tau_i(f_{\eps'}(b))\leq c,$$ which is a contradiction. This proves \eqref{3031}. Thus, $$d_{\tau}(b_i)=\sup_{\eps'>0}f_{\eps'}(b_i)\geq \tau(f_{\eps'}(b_i))>c.$$ Note that $c$ is independent of $C_i$.

Therefore, one may assume that $C_i$ is large enough such that $\mathrm{rc}(C_i)<c/(n+1)$. It follows from \eqref{eq310} that $$d_\tau((a_i-\eps)_+)+\frac{1}{n+1}d_\tau(b_i)\leq d_\tau(b_i)$$ for any tracial state $\tau$ on $C_i$. Since $d_\tau(b_i)>c>0$ for all $\tau$, one has $$d_\tau((a_i-\eps)_+)+\frac{c}{n+1}\leq d_\tau(b_i)$$ for any tracial state $\tau$ on $C_i$.

Since $\mathrm{rc}(C_i)<c/(n+1)$, one has that $$(a_i-\eps)_+\precsim b_i.$$ Note that $\norm{(a-\eps)_+-(a_i-\eps)_+}\leq3\eps$, one has that $(a-4\eps)_+\precsim(a_i-\eps)_+$, and hence $$(a-4\eps)_+\precsim(a_i-\eps)_+\precsim b_i\precsim b.$$ Since $\eps$ is arbitrary, one has that $a\precsim b$, as desired.
\end{proof}

\begin{thm}\label{MD0-Comp}
Let $A$ be a simple unital AH-algebra with diagonal maps. If $A$ has mean dimension zero, then it has strict comparison of positive elements.
\end{thm}
\begin{proof}
It follows directly from Theorem \ref{LOA}, Lemma \ref{LOA-RC}, and Theorem 3.15 of \cite{Toms-PLMS}.
\end{proof}

The next theorem is an analog to Theorem 3.4 of \cite{Toms-SDG} for local approximations, and the proof is similar.
\begin{lem}\label{LOA-AD}
Let $A$ be a simple separable unital C*-algebra satisfying the following property: For any finite subset $\F\subseteq A$ and any $\eps>0$, there is a unital sub-C*-algebra $C$ such that $\F\subseteq_{\eps} C$ and $C$ is a recursive sub-homogeneous C*-algebra with dimension ratio less than $\eps$. Then, for any strictly positive continuous affine function $f$ on $\mathrm{T}(A)$ and any $\eps>0$, there exist positive elements $a\in \mathrm{M}_{n}(A)$ for some $n$ such that $$\abs{f(\tau)-d_\tau(a)}<\eps\quad\forall \tau\in\mathrm{T}(A).$$
\end{lem}
\begin{proof}
Since $f$ is strictly positive, there exists $\delta>0$ such that $f(\tau)>\delta$ for all $\tau\in\mathrm{T(A)}$. Without loss of generality, one may assume that $\eps<\delta/8$. 

Let $(\F_i)$ be an increasing sequence of finite subsets of $A$ with dense union, and let $(C_i)$ be the sequence of corresponding sub-C*-algebras with $\F_i$ approximates within $\eps/i$ of $C_i$. Since $A$ is simple, one may assume that the dimension of the irreducible representation of $C_i$ is at least $i^2$ and the dimension ratio of $C_i$ is less that $\eps/i$.

Denot by $\iota_i$ the embedding of $C_i$, and  denoted by $\iota'_i: \mathrm{Aff}(\mathrm{T}(C_i))\to \mathrm{Aff}(\mathrm{T}(A))$ the map induced by $\iota_i$. Since $A$ can be locally approximated by $\{C_i\}$, it is well known that the set ${\bigcup_i\iota'_i(\mathrm{Aff}(\mathrm{T}(C_i)))}$ is dense in $\mathrm{Aff}(\mathrm{T}(A))$. Therefore, for sufficiently large $k$, the set $\iota'_k(\mathrm{Aff}(\mathrm{T}(C_k)))$ approximately contains $f$ up to $\eps/2$, and thus, one may assume that $\iota'_i(\mathrm{Aff}(\mathrm{T}(C_i)))$ approximately contains $f$ up to $\eps/2$ for all $i$.

For each $i$, pick $f_i\in \mathrm{Aff}(\mathrm{T}(C_i))$ such that $\norm{\iota'_i(f_i)-f}<\eps/2$ . Since $f$ is strictly positive, by an asymptotic argument, one may assume further that the sub-C*-algebras $\{C_i\}$ and $\{f_i\}$ are chosen in a way so that $f_i(\tau)>\delta/2$ for all $k$ and for all $\tau\in\mathrm{T}(C_i).$ 

Fix $C_i$ with $i>\max\{8+2/\eps, 8/\delta\}$. Denote by $\{X_1, ..., X_l\}$ the base spaces of $C_i$, and denote by $X=\coprod_{k=1}^l X_k$. Note that each $x\in X$ induces a tracial state of $C_i$, and denote it by $\tau_x$.

For each $1\leq k\leq l$, consider the functions $h_k(x)=\lceil n_xf_i(\tau_x)\rceil-2$ and $g_k(x)=\lfloor n_xf_i(\tau_x)\rfloor-\eps n_x/2-3$, where $x\in X_k$, and $n_x$ is the matrix size of $C_i$ at $x$, and $d_x$ is the dimension of the connected component of the base space $X$ containing $x$. Then the functions $h_k$ and $g_k$ are lower semicontinuous and upper semicontinuous respectively. Note that 
$$h_k(x)-g_k(x)=2+\frac{\eps n_x}{2}\geq 4d_x.$$ 

Assert that there exists $a\in C_i$ such that $$g_k(x)\leq\mathrm{rank}(\pi_k(a)(x))\leq h_k(x),$$ where $\pi_k$ is the restriction of to $X_k$.
 
If $l=0$, since $f_i(\tau_x)>\delta/2$, $d_x/n_x<\eps/2$, $n_x>i^2$, and $\eps<\delta/8$, one has that $$g(x)>n_x\delta/2-4d_x-2>i^2\delta/4-2.$$ Thus the function $g(x)$ is positive.  
Hence, by Proposition 2.9 of \cite{Toms-SDG} (with $Y=\textrm{\O}$), there exists a positive element $a\in \mathrm{M}_n(C_i)$ for some $n$ such that $$g(x)\leq\mathrm{rank}(a(x))\leq h(x),$$ and hence $$\abs{d_{\tau_x}(a)-f_i(\tau_x)}\leq \abs{\frac{g(x)}{n_x}-\frac{h(x)}{n_x}}<\frac{\eps}{2}.$$ 

Assume that the assertion is true for $l=K$. There exists $a_K\in R^{(K)}$ such that $$g_k(x)\leq\mathrm{rank}(\pi_k(a_K)(x))\leq h_k(x), \quad\ 0\leq k\leq K.$$ Consider $b=\varphi(a_K)\in \textrm{M}_n(\textrm{C}(Y_{K+1}))$. Then, for any $y\in Y_{K+1}$, $$g_{K+1}(y)\leq \mathrm{rank}(b(y))\leq h_{K+1}(y).$$ Indeed, there are points $\{y_1,..., y_m\}\subseteq \coprod_{j=1}^K X_j$ such that the representation $\pi_y$ of $R^{(k)}$ is unitarily equivalent to the direct sum of $\pi_{y_i}$. Therefore, $$n_y\tau_y=\sum n_{y_j}\tau_{y_j},$$ and hence
\begin{eqnarray*}
\textrm{rank}(b(y))&=&\textrm{rank}(a(y_j))\\
&=&\sum h_j(y_j)\\
&\leq& \sum (\lceil n_{y_j}f_i(\tau(y_j))\rceil -2)\\
&\leq& \lceil n_yf_i(\tau_y)\rceil-2=h_{K+1}(y).
\end{eqnarray*}
The same argument shows that $g_{k+1}(y)\leq\textrm{rank}(b(y))$. Therefore, by Proposition 2.9 of \cite{Toms-SDG}, there is a function $a'\in  \textrm{M}_n(\textrm{C}(X_{K+1}))$ such that the restriction of $a'$ to $Y_{K+1}$ equals to $b$ and $$g_{K+1}(x)\leq \mathrm{rank}(a(x))\leq h_{K+1}(x), \quad\forall x\in X_{K+1}.$$ Thus, the assertion holds for $l=K+1$, and hence for $C_i$ with arbitrary length.

Therefore it follows from the assertion that $\abs{d_{\tau}(a)-f_i(\tau)}<\eps/2$ for all $\tau\in\mathrm{T}(C_i)$, and hence for all $\tau\in\mathrm{T}(A)$, one has
\begin{eqnarray*}
\abs{d_{\tau}(a)-f(\tau)}&\leq&\abs{d_{\tau}(a)-\iota'_i(f_i)(\tau)}+\abs{\iota'_i(f_i)(\tau)-f(\tau)}\\
&<&\abs{d_{\tau\circ\iota_i}(a)-f_i(\tau\circ\iota_i)}+\eps/2\\
&<&\eps,
\end{eqnarray*}
as desired.
\end{proof}

\begin{cor}\label{rob}
{Any simple unital AH-algebra with diagonal maps and mean dimension zero is an AH-algebra without dimension growth.}
%The class of simple unital AH-algebras with diagonal maps and mean dimension zero are AH-algebras without dimension growth.
\end{cor}
\begin{proof}
By Theorem \ref{MD0-Comp} and Lemma \ref{LOA-AD}, the Cuntz semigroup of an AH-algebra with mean dimension zero is $\mathrm{V}(A)\sqcup\mathrm{SAff}(\mathrm{T}(A))$. It then follows from \cite{Winter-Z-stable-02} that it is $\mathcal Z$-stable. By \cite{Winter-Z}, \cite{Lin-Asy}, \cite{L-N}, and \cite{Lnclasn}, {this C*-algebra is an AH-algebra} without dimension growth.
\end{proof}

\begin{cor}\label{cor-general}
If a unital simple AH-algebra with diagonal maps has at most countably many extremal tracial states, or if there is $M>0$ such that $\rho^{-1}(\kappa)$ has at most $M$ extreme points for any $\kappa\in\mathrm{S}(\Kzero(A))$, then it is an AH-algebra without dimension growth.
\end{cor}
\begin{proof}
It follows from Corollary \ref{rob}, Theorem \ref{RR0-SBP}, and Corollary \ref{lisa}.
\end{proof}

In the following, let us consider the tensor product $A\otimes A$, where $A$ is any simple AH-algebra with diagonal maps (without any assumptions on dimension growth or mean dimension). As another application of Lemma \ref{LOA-RC} and Lemma \ref{LOA-AD}, we shall show that $A\otimes A$ is always an AH-algebra with slow dimension growth.
\begin{lem}\label{approx-general}
Let $A$ be a simple unital AH-algebra with diagonal maps. Then, for any finite subset $\mathcal F\subseteq A$ and any $\eps>0$, there is a unital sub-C*-algebra $C\subseteq A$ with $C\cong\bigoplus_{i=1}^h p_i\MC{r_i}{\Omega_i}p_i$ such that $\mathcal F\subseteq_\eps C$ and 
$$\frac{\mathrm{dim}(\Omega_i)}{\mathrm{rank}(p_i)}\cdot\frac{1}{\min\{\mathrm{rank}(p_i);\ i=1, ..., h\}}<\eps,\quad\forall i.$$
\end{lem}

\begin{proof}
Suppose that $A=\varinjlim(A_i, \varphi_{i})$, where $A_i=\bigoplus_{l=1}^{h_i}\MC{r_{i, l}}{X_{i, l}}$, and assume that $\mathcal F\subseteq A_1$. Choose a finite open cover $\alpha$ of $X_1:=\bigcup_l X_{1, l}$ such that for any $U\in\alpha$, one has
$$\norm{f(x)-f(y)}<\eps,\quad\forall f\in\mathcal F,\ \forall x, y\in U.$$

Let $\{\phi_U: X_1\to [0, 1];\ U\in\alpha\}$ be a partition of unity which is subordinate to $\alpha$. Consider the $\mathrm{ord}(\alpha)$-dimensional simplicial complex $\Delta$ constructed as follows: The vertices of $\Delta$ correspond to the elements of $\alpha$, the $s$-dimensional simplices correspond to all $U_1, ..., U_s$ with ${\bigcap U_i\neq \textrm{\O}}$. The point in each simplex $\{U_1, ..., U_s\}$ can be parameterized as 
$$\sum_{i=1}^s\lambda_1[U_i]\quad\textrm{with $\lambda_i\geq 0$ and $\sum_i\lambda_i=1.$}$$
Define a map $\xi: X_1\to \Delta$ by $$x\mapsto \sum_{x\in U\in\beta}\phi_U(x)[U],$$ and it induces a *-homomorphism 
$$\Xi: \bigoplus_l\MC{r_{1, l}}{\Delta_{l}}\to {\bigoplus_l\MC{r_{1, l}}{X_{1, l}}}=A_1,$$
where $\Delta_l$ is the connected component of $\Delta$ corresponding to $X_{1, l}$. Then the same argument as that of Theorem \ref{LOA} shows that any element of $\mathcal F$ can be approximated by the elements of the image of $\Xi$ up to $\eps$. 

Note that $\mathrm{dim}(\Delta)=\mathrm{ord}(\alpha)<\infty$, and $A$ is assumed to be simple. One may assume that $A_2$ is far enough so that the multiplicities $\{m_{1, 2}^{l, k};\ 1\leq l\leq h_1, 1\leq k\leq h_2\}$ of $\varphi_{1, 2}$ are sufficiently large such that
\begin{equation}\label{larg-mult}
\frac{\mathrm{dim}(\Delta)}{\min\{r_{1, l};\ 1\leq l\leq h_1\}}\cdot\frac{1}{{\min\{m_{1, 2}^{l, k};\ 1\leq l\leq h_1, 1\leq k\leq h_2\}}}<\eps.
\end{equation}

For each $1\leq k\leq h_2$, consider $\varphi_{1, 2}^k$, the restriction of $\varphi_{1, 2}$ to $A_1$ and $\MC{r_{2, k}}{X_{2, k}}$. Then one has a factorization 
\begin{displaymath}
\xymatrix{A_1\ar[r]^{\Psi} & B\ar[r] & \MC{r_{2, k}}{X_{2, k}},}
\end{displaymath} 
where $$B=\bigoplus_{l=1}^{h_1}\MC{m_{1, 2}^{l, k}r_{1, l}}{\underbrace{X_{1, l}\times\cdots\times X_{1, l}}_{m_{1, 2}^{l, k}}}.$$ 
Indeed, the map $\Psi: A_1\to B$ is induced by
$$\MC{r_{1, l}}{X_{1, l}}\ni f\mapsto \mathrm{diag}\{f\circ\pi_1, ..., f\circ\pi_{m_{1, 2}^{l, k}} \}\in \MC{m_{1, 2}^{l, k} r_{1, l}}{\underbrace{X_{1, l}\times\cdots\times X_{1, l}}_{m_{1, 2}^{l, k}}},$$
where $\pi_1$, ..., $\pi_{m_{1, 2}^{l, k}}$ are the coordinator projections from $X_{1, l}\times\cdots\times X_{1, l}$ to each copy of $X_{1, l}$, and the map $B\to \MC{r_{2, k}}{X_{2, k}}$ is induced by 
$$\bigoplus_{l=1}^{h_1}\MC{m_{1, 2}^{l, k}r_{1, l}}{\underbrace{X_{1, l}\times\cdots\times X_{1, l}}_{m_{1, 2}^{l, k}}}\ni (f_1, ..., f_{h_1}) \mapsto \mathrm{diag}\{f_1\circ\Lambda_1, ..., f_{h_1}\circ\Lambda_{h_1}\}\in \MC{r_{2, k}}{X_{2, k}},$$
where each continuous map $\Lambda_{l}: X_{2, k}\to \underbrace{X_{1, l}\times\cdots\times X_{1, l}}_{m_{1, 2}^{l, k}}$ is defined by $$\Lambda_l(x)=(\lambda_{1, 2}^{l, k}(1)(x), ..., \lambda_{1, 2}^{l, k}(m_{1, 2}^{l, k})(x)),$$
where $\lambda_{1, 2}^{l, k}(1)$, ..., $\lambda_{1, 2}^{l, k}(m_{1, 2}^{l, k})$ are eigenvalue maps from $X_{2, k}$ to $X_{1, l}$.

Consider the C*-algebra $B$, and define the map
$$\tilde{\Xi}: \bigoplus_{l=1}^{h_1}\MC{m_{1, 2}^{l, k}r_{1, l}}{\underbrace{\Delta_{1, l}\times\cdots\times \Delta_{1, l}}_{m_{1, 2}^{l, k}}} \to \bigoplus_{l=1}^{h_1}\MC{m_{1, 2}^{l, k}r_{1, l}}{\underbrace{X_{1, l}\times\cdots\times X_{1, l}}_{m_{1, 2}^{l, k}}}$$ 
by
$$\MC{m_{1, 2}^{l, k}r_{1, l}}{\underbrace{\Delta_{1, l}\times\cdots\times \Delta_{1, l}}_{m_{1, 2}^{l, k}}}\ni f\mapsto f\circ (\xi_l, ..., \xi_l)\in \MC{m_{1, 2}^{l, k}r_{1, l}}{\underbrace{X_{1, l}\times\cdots\times X_{1, l}}_{m_{1, 2}^{l, k}}},$$
where $\xi_l$ is the restriction of $\xi$ to $X_{1, l}$.

Then one has the commutative diagram
\begin{displaymath}
\xymatrix{
A_1 \ar[r]^{\Psi}  &   B \\
\bigoplus_l\MC{r_{1, l}}{\Delta_{l}} \ar[u]^{\Xi} \ar[r] & \bigoplus_{l=1}^{h_1}\MC{m_{1, 2}^{l, k}r_{1, l}}{\underbrace{\Delta_{1, l}\times\cdots\times \Delta_{1, l}}_{m_{1, 2}^{l, k}}} \ar[u]_{\tilde{\Xi}},
}
\end{displaymath}
where the bottom horizontal map is induced by the coordinator projections. In particular, the image of the map $\Psi\circ\Xi$ is contained in the image of $\tilde{\Xi}$, and hence the finite subset $\mathcal F$ is approximately contained in the image of $\tilde{\Xi}$ up to $\eps$.

Denoted by $$C=\bigoplus_{k}\bigoplus_{l=1}^{h_1}\MC{m_{1, 2}^{l, k}r_{1, l}}{\Omega_{l, k}}$$ the image of $\tilde{\Xi}$, where $\Omega_{l, k}$ is a closed subset of in $\underbrace{\Delta_{1, l}\times\cdots\times \Delta_{1, l}}_{m_{1, 2}^{l, k}}$ (the image under $(\xi_l, ..., \xi_l)$). Hence, one has that $\mathcal F\subseteq_{\eps} C$ and
$$\mathrm{dim}(\Omega_{l, k}) \leq m_{1, 2}^{l, k}\cdot \mathrm{dim}(\Delta).$$
Therefore, by \eqref{larg-mult}, one has
\begin{eqnarray*}
&&\frac{\mathrm{dim}(\Omega_{l, k})}{m_{1, 2}^{l, k} r_{1, l}} \cdot\frac{1}{\min\{m_{1, 2}^{l, k} r_{1, l};\ 1\leq l\leq h_1, 1\leq k\leq h_2\}}\\
&\leq& \frac{\mathrm{dim}(\Delta)}{r_{1, l}} \cdot\frac{1}{\min\{m_{1, 2}^{l, k} r_{1, l};\ 1\leq l\leq h_1, 1\leq k\leq h_2\}}\\
&<&\eps.
\end{eqnarray*}
%$$\frac{\mathrm{dim}(\Omega_{l, k})}{ (m_{1, 2}^{l, k} r_{1, l})^2} \leq \frac{m_{1, 2}^{l, k}\cdot \mathrm{dim}(\Delta)}{ (m_{1, 2}^{l, k} r_{1, l})^2} \leq  \frac{ \mathrm{dim}(\Delta)}{ r_{1, l}} \cdot \frac{1}{m_{1, 2}^{l, k} } <\eps,$$ 
as desired.

Suppose, in general, that $A=\varinjlim(A_i, \varphi_{i})$, where $A_i=\bigoplus_{l=1}^{h_i}p_{i, l}\MC{r_{i, l}}{X_{i, l}}p_{i, l}$ and each $p_{i, l}$ is a projection in $\MC{r_{i, l}}{X_{i, l}}$. Then there is a simple unital AH-algebra $D$ with diagonal maps and a projection $p\in D$ such that $A\cong pDp$ (in fact, the AH-algebra $D$ is the limit of $(\bigoplus_{l=1}^{h_i}\MC{r_{i, l}}{X_{i, l}}, \tilde{\varphi_i}$), where $\tilde{\varphi_i}$ is induced by $\varphi_i$). Since $B$ is simple, there is $\delta>0$ such that $\tau(p)>\delta$ for all $\tau\in\mathrm{T}(B)$.

By the conclusion above, there is a unital sub-C*-algebra $C\subseteq B$ such that $C\cong \bigoplus_{k=1}^h \MC{r_k}{\Omega_k}$, $\mathcal F\subseteq_{\eps} C$ and 
\begin{equation}\label{contr-big}
\frac{\mathrm{dim}(\Omega_k)}{r_k}\cdot\frac{1}{\min\{r_k;\ 1\leq k\leq h\}}<\eps\delta^2,\quad\forall k.
\end{equation}
Without loss of generality, one may assume that each element of $\mathcal F$ has norm at most one, and $p\in C$ (otherwise, let $p\in\mathcal F$, and then apply a conjugation of $C$ by a unitary which is very close to $1_B$). Moreover, by an asymptotical sequence argument, one may also assume that 
\begin{equation}\label{lb-trace}
\tau(p)\geq\delta,\quad\forall \tau\in\mathrm{T}(C).
\end{equation}

Then one has that $pCp\subseteq pDp=A$ and $$\mathcal F\subseteq_{\eps} pCp\cong \bigoplus_k p_k\MC{r_k}{\Omega_k}p_k,$$ where $p_k$ is the direct summand of $p$ in $\MC{r_k}{\Omega_k}$. Assume that $\min\{\mathrm{rank}(p_k);\ 1\leq k\leq h\}$ is obtained on $p_j$. By \eqref{lb-trace}, one has 
$$\frac{\mathrm{rank}(p_k)}{r_k}\geq\delta,\quad\forall k,$$
and therefore, for any $k$, by \eqref{contr-big}, one has
\begin{eqnarray*}
&&\frac{\mathrm{dim}(\Omega_k)}{\mathrm{rank}(p_k)}\cdot\frac{1}{\min\{\mathrm{rank}(p_k);\ k=1, ..., h\}}\\
&=& \frac{\mathrm{dim}(\Omega_k)}{r_k}\cdot \frac{r_k}{\mathrm{rank}(p_k)}\cdot\frac{r_j}{\mathrm{rank}(p_j)}\cdot\frac{1}{r_j}\\
&\leq &\delta^2 \cdot \frac{\mathrm{dim}(\Omega_k)}{r_k}\cdot \frac{1}{r_j}\\
&\leq &\delta^2 \cdot \frac{\mathrm{dim}(\Omega_k)}{r_k}\cdot \frac{1}{\min\{r_k;\ 1\leq k\leq h\}}\\
&<&\delta^2\cdot \frac{\eps}{\delta^2}=\eps,
\end{eqnarray*}
as desired.
\end{proof}

\begin{thm}\label{tensor}
Let $A$ be an arbitrary simple unital AH-algebra with diagonal maps (without any assumptions on dimension growth or mean dimension). Then $A\otimes A$ is an AH-algebra without dimension growth. In other words, the Toms' growth rank (\cite{Toms-GR}) of $A$,  denoted by $\mathrm{gr}(A)$, is at most two.
\end{thm}
\begin{proof}
Let us show that for any finite subset $\F\subseteq A\otimes A$ and any $\eps>0$, there is a homogeneous C*-algebra $D\subseteq A\otimes A$ such that $\F\subseteq_\eps D$ and the dimension ratio of $D$ is at most $\eps$. Then the theorem follows from Lemma \ref{LOA-RC} and Lemma \ref{LOA-AD}.

Without loss of generality, one may assume that $$f_i=\sum_{j=1}^{l_i}f_{i, j}^{(1)}\otimes f_{i, j}^{(2)}$$ for each $f_i\in\F$. Set $l=\max_i\{l_i\}$, $s=\max_{i, j}\{\norm{f^{(1)}_{i, j}}, \norm{f^{(2)}_{i, j}}\}$, and $$\mathcal G=\{f_{i, j}^{(1)}, f_{i, j}^{(2)};\ f_i\in\F\}.$$ 

By Lemma \ref{approx-general}, there is $C\subseteq A$ with $C\cong \bigoplus_{k=1}^h p_k\MC{r_k}{\Omega_k}p_k$ such that $\mathcal G\subseteq_{\eps/ls^2} C$ and
\begin{equation}\label{contr-big-1}
\frac{\mathrm{dim}(\Omega_k)}{\mathrm{rank}(p_k)}\cdot\frac{1}{\min\{\mathrm{rank}(p_i);\ i=1, ..., h\}}<\eps/2,\quad\forall k.
\end{equation}
One then has that $\F\subseteq_{\eps} C\otimes C$. Moreover,
$$C\otimes C \cong \bigoplus_{i, j} (p_i\otimes p_j) \MC{r_ir_j}{\Omega_i\times\Omega_j} (p_i\otimes p_j).$$
By \eqref{contr-big-1}, for any $1\leq i, j\leq h$, one has
\begin{eqnarray*}
\frac{\mathrm{dim}(\Omega_i\times\Omega_j)}{\mathrm{rank}(p_i \otimes p_j)} & = &\frac{\mathrm{dim}(\Omega_i)+\mathrm{dim}(\Omega_j)}{\mathrm{rank}(p_i) \mathrm{rank}(p_j)}\\
&\leq & \frac{\mathrm{dim}(\Omega_i)}{\mathrm{rank}(p_i)}\cdot\frac{1}{\mathrm{rank}(p_j)} + \frac{\mathrm{dim}(\Omega_j)}{\mathrm{rank}(p_j)}\cdot\frac{1}{\mathrm{rank}(p_i)}\\
&<&\frac{\eps}{2}+\frac{\eps}{2}=\eps.
\end{eqnarray*}
Therefore $D:=C\otimes C$ is the desired sub-C*-algebra.
\end{proof}

\section{Radius of comparison for AH-algebras with diagonal maps}\label{comp-radius}

%\subsection{A lower bound for mean dimension} 
In this section, a lower bound for the mean dimension is given in the term of radius of comparison. 

\begin{lem}\label{LOA-RC1}
Let $A$ be a simple exact C*-algebra satisfying the following properties: 
\begin{enumerate}
\item there is $r\geq 0$ such that for any finite subset $\F\subseteq A$ and any $\eps_1, \eps_2>0$, there is a sub-C*-algebra $C$ such that $\F\subseteq_{\eps_1} C$ and $\mathrm{rc}(C)<r+\eps_2$;
\item for any $0\leq s\leq 1$ and any $\eps>0$, there exists a projection $z$ with $\abs{\tau(z)-s}<\eps$, $\forall\tau\in\mathrm{T}(A)$.
\end{enumerate}
Then one has that $\mathrm{rc}(A)\leq r$.
\end{lem}
\begin{proof}
Since $A$ is exact (and hence all sub-C*-algebras $C_i$), any lower semicontinuous dimension function on $W(A)$ is induced by a tracial state.

Let $a$ and $b$ be positive elements in $A$ with $$d_{\tau}(a)+r'<d_{\tau}(b),\quad\forall \tau\in\mathrm{T}(A)$$ for some $r'>r$. Let us show that $a \precsim b$.

Choose a projection $z$ such that $r<\tau(z)<r'$, $\forall\tau\in\textrm{T}(A)$. Since $d_{\tau}(a\oplus z)<d_{\tau}(b)$ for any tracial state of $A$, there is a rational number $0<c<1$ such that $d_{\tau}(a\oplus z)<cd_{\tau}(b)$ for all tracial state $\tau$ on $A$, and hence there are natural numbers $m<m'$ such that $$m'd_{\tau}(a\oplus z)<md_{\tau}(b).$$ By Proposition 3.1 of \cite{RorUHF-II}, there exists $n\in\mathbb N$ an $d\in W(A)$ such that
\begin{equation}\label{4030}
n(m'[a\oplus z])+ d\leq nm[b]+ d
\end{equation} 
in $W(A)$. Applying \eqref{4030} $k$ times, one has $$km'n([a\oplus z])+ d\leq kmn[b]+ d,\quad\forall k\in \mathbb N.$$ Since $W(A)$ is simple, there exists $l$ such that $d\leq l[b]$, and hence $$km'n[a\oplus z]\leq (kmn+l)[b]\quad\forall k\in\mathbb N.$$ Choose $k$ sufficiently large so that $km'n>kmn+l$, and set $N=km'n-1$, one has $$(N+1)[a\oplus z]\leq N[b].$$

One can then follows the same proof as that of Lemma \ref{LOA-RC}. Fix $\eps>0$ for the time being. It follows from Lemma \ref{localize} that there exist $C_i$ and $a_i, b_i, z_i\in C_i$ such that $z_i$ is a projection,  $$\norm{a_i-a}\leq\eps,\quad \norm{b_i-b}\leq\eps\quad \textrm{and}\quad b_i\precsim b, $$ and 
\begin{equation}\label{eq410}
\bigoplus_{N+1}((a_i-\eps)_+\oplus z_i)\precsim \bigoplus_Nb_i
\end{equation} 
in $C_i$. Moreover, using the simplicity of $A$, the compactness of the simplex of tracial states of $A$, and the lower-semicontinuity of $d_\tau$, one may assume that there is a strictly positive number $c$ which is independent of $C_i$ such that 
\begin{equation}\label{4033}
d_\tau(b_i)>c\quad\textrm{and}\quad r<\tau(z_i)<r'
\end{equation} for any tracial state $\tau$ of $C_i$. 

Indeed, for any $\eps'>0$, consider $f_{\eps'}(b_i)$. Fix $\eps'$, and it is clear that $f_{\eps'}(b_i)\to f_{\eps'}(b)$. Since $A$ is simple, there exists $c>0$ such that $$\tau(f_{\eps'}(b))>c$$ for any tracial state $\tau$ of $A$.  Then, there exists a sub-C*-algebra $C_i$ such that 
\begin{equation}\label{4031}
\tau(f_{\eps'}(b_i))>c
\end{equation} for all tracial state $\tau$ of $C_i$. If this were not true, there is a sequence of $C_i$ with dense union in $A$, positive elements $b_i\in C_i$ and a sequence of tracial state $\tau_i$ on $C_i$ such that $$\norm{f_{\eps'}(b_i)-f_{\eps'}(b)}\leq 1/2^i\quad\textrm{and}\quad{\tau_i}(f_{\eps'}(b_i))\leq c.$$ Extend $\tau_i$ to a state of $A$, and pick an accumulation point $\tau_\infty$, and assume that $\tau_i\to\tau_\infty$. One then has $$\tau_i(f_{\eps'}(b))\leq c+1/2^i,$$ and  $$\tau_\infty(f_{\eps'}(b))=\lim_{i\to\infty} \tau_i(f_{\eps'}(b))\leq c,$$ which is a contradiction. This proves \eqref{4031}. Thus, $$d_{\tau}(b_i)=\sup_{\eps'>0}f_{\eps'}(b_i)\geq \tau(f_{\eps'}(b_i))>c.$$ Note that $c$ is independent of $C_i$. A similar argument shows that $r<\tau(z_i)<r'$.

Therefore, one may assume that $C_i$ is large enough such that $\mathrm{rc}(C_i)<r+c/(N+1)$. It follows from \eqref{eq410} that $$d_\tau((a_i-\eps)_+)+d_\tau(z_i)+\frac{1}{N+1}d_\tau(b_i)\leq d_\tau(b_i)$$ for any tracial state $\tau$ of $C_i$. By \eqref{4033}, $c_i=\inf_\tau\{d_\tau(b_i)\}\geq c$ is strictly positive, and hence one has $$d_\tau((a_i-\eps)_+)+r+\frac{c}{N+1}\leq d_\tau(b_i)$$ for any tracial state $\tau$ of $C_i$.

Since $\mathrm{rc}(C_i)<r+c/(N+1)$, one has that $$(a_i-\eps)_+\precsim b_i.$$ Note that $\norm{(a-\eps)_+-(a_i-\eps)_+}\leq3\eps$, one has that $(a-4\eps)_+\precsim(a_i-\eps)_+$, and hence $$(a-4\eps)_+\precsim(a_i-\eps)_+\precsim b_i\precsim b.$$ Since $\eps$ is arbitrary, one has that $a\precsim b$, as desired.
\end{proof}

\begin{thm}
If an AH-algebra $A$ with diagonal maps has mean dimension $\gamma$, then $\mathrm{rc}(A)\leq \gamma/2$.
\end{thm}
\begin{proof}
Since the dimensions of the irreducible representations of $A_i$ go to $\infty$ {as $i\to\infty$}, the C*-algebra $A$ satisfies Condition (2) of Lemma \ref{LOA-RC1}. Let us show that $A$ also satisfies Condition (1) of Lemma \ref{LOA-RC1}. 

Since $A$ has mean dimension $\gamma$, by Theorem \ref{LOA}, for any finite subset $\F\subseteq A$ and any $\eps_1>0$ and $\eps_2>0$, there exists a unital sub-C*-algebra $$C\cong\bigoplus\MC{n_i}{\Omega_i}\subseteq A$$ such that $\F\subseteq_{\eps_1} C$, and $$\frac{\mathrm{dim}{\Omega_i}}{n_i}<\gamma+\eps_2.$$

By Corollary 5.2  of \cite{Toms-Comp-DS}, $$\mathrm{rc}(C)\leq \gamma/2+\eps_2.$$ It then follows from Lemma \ref{LOA-RC1} that $\mathrm{rc}(A)\leq \gamma/2$.
\end{proof}

\section{Cuntz mean dimension for AH-algebras with generalized diagonal maps}
\begin{defn}\label{Bcover}
A branched open cover $\tilde{\alpha}$ of a compact Hausdorff space $X$ consists pairs $$\{(U_\lambda, \kappa_\lambda);\ \lambda\in\Lambda\}$$ with $U_\lambda$ an open subset of $X$ and {$\kappa_\lambda$} a Murray-von Neumann {equivalence} class of a projection in $\mathrm{C}(\overline{U}_\lambda, \mathcal K)$ such that $\bigcup_{\lambda\in\Lambda} U_\lambda=X.$ Note $\{U_\lambda\}$ is an ordinary cover of $X$, and denote it by $\alpha$.

For any $U\in\alpha$, define $$\mathcal K_U=\{\kappa_V;\ (V, \kappa_V)\in\tilde{\alpha},\ U=V \}.$$ Then the Murray von Neumann equivalence relation induces a partition $$\mathcal K_U=\mathcal E_U(1)\sqcup\cdots \sqcup \mathcal E_U(n_U).$$ For any $\mathcal E_U(i)$, define $$\mathrm{rank}(\mathcal E_U(i))=\mathrm{rank}(p),$$ where $p$ is any representative of $\mathcal E_U(i)$.

Define the multiplicity of $U$ by $$\mathrm{mul}(U):=\min\{\abs{\mathcal E_U(i)}\cdot\mathrm{rank}(\mathcal E_U(i));\ 1\leq i\leq n_U\},$$ and define the multiplicity of $\tilde{\alpha}$ to be $$\mathrm{mul}(\tilde{\alpha}):=\min\{\textrm{mul}(U);\ U\in\alpha\}.$$
\end{defn}

\begin{defn}\label{Bintersection}
Let $\tilde{\alpha}$ and $\tilde{\beta}$ be two branched covers. Define $\tilde{\alpha}\vee\tilde{\beta}$ to be the branched cover consisting of $$\{(U\cap V, \kappa_{U\cap V});\ U\in\alpha,\ V\in\beta,\ \kappa_{U\cap V}=\kappa_U\ \textrm{or}\ \kappa_{U\cap V}=\kappa_V\}.$$
\end{defn}

\begin{defn}\label{Binduce}
Let $\tilde{\alpha}$ be a branched cover of $X$, and let $\beta$ be any ordinary cover with $\beta\succ \alpha$. For any $W\in\beta$, consider $$\tilde{\alpha}_W=\{(U, \kappa_U);\ (U, \kappa_U)\in\tilde{\alpha}\ \textrm{and}\ W\subseteq U\},\quad \textrm{and}\quad \mathcal K^{\tilde{\alpha}}_W=\{(\kappa_U)|_W;\ (U, \kappa_U)\in\tilde{\alpha}_W\}.$$
Then $$\mathrm{Ind}^{\tilde{\alpha}}_\beta:=\{(W, \kappa_W);\ W\in\beta,\ \kappa_W\in\mathcal{K}_W^{\tilde{\alpha}}\}$$ is a branched cover of $X$ induced by $\alpha$ based on $\beta$.
\end{defn}

\begin{NN}
Consider the homogeneous C*-algebra $A=p(\mathrm{C}(X)\otimes\mathcal K)p$, where $X$ is a compact Hausdorff space, and $p$ is a projection. Let $w$ be an element in $W(A)$ such that there exists an open set $U$ and $d\in \mathbb N$ such that 
\begin{equation}\label{type0}
d_{\tau_x}(w)=\left\{\begin{array}{ll}
d, &\textrm{if}\ x\in U,\\
0, & \textrm{otherwise},
\end{array}
\right.
\end{equation}
where $\tau$ is the Dirac measure concentrated on $x$.
\end{NN}

\begin{lem}
Let $w$ be the Cuntz semigroup element as above. Then, for any open subset $U'\subset U$ with $\overline{U'}\subset U$, there is a vector bundle $\xi$ on $\overline{U'}$ such that $w|_{U'}$ is induced by $\xi$.
\end{lem}

\begin{proof}
Let $a(x)\in \mathrm{M}_n({A^+})$ be a representative of $w$. Then one has that $\mathrm{supp}(a)=\overline{U}$, and for any $x\in U$, $\mathrm{rank}(a(x))=d$. Then, $\mathrm{Im}(a(x))$ is the desired vector bundle.
\end{proof}

\begin{defn}
An element $w$ in $W(A)$ is said to be type 0 if it satisfies \eqref{type0}, and $U$ is called an open support of $w$. 
\end{defn}

\begin{NN}
Let $w_1, ..., w_n$ be elements of type 0 in $W(A)$, and denoted by $U_1, ..., U_n$ the their open supports respectively. If $\{U_1, ..., U_n\}$ forms an open cover of $X$, then $w_1, ..., w_n$ induce a branched cover of $X$ in the natural way. 
\end{NN}

\begin{defn}[\cite{Vill-sr}]
A homomorphism: $\varphi: \textrm{C}(X)\otimes\mathcal K\to \mathrm{C}(Y)\otimes\mathcal K$ is called generized diagonal if there exist $k\in\mathbb N$ and maps $\lambda_1, ..., \lambda_k:  Y \to X$ and mutually orthogonal projections $p_1, ..., p_k$ in $\mathrm{C}(Y)\otimes\mathcal K$ such that $\varphi=(\mathrm{id}_{\mathrm{C}(X)}\otimes\vartheta)\circ(\tilde{\varphi}\otimes\mathrm{id_{\mathcal K}})$ where 
$$\tilde{\varphi}:\mathrm{C}(X)\to\mathrm{C}(Y)\otimes\mathcal{K}: \quad f\mapsto\sum_{i=1}^k(f\circ\lambda_i)p_i$$ 
and $\vartheta:\mathcal K\otimes\mathcal K\to\mathcal K$ is an isomorphism. In this case, one {says that} the map $\varphi$ comes from {the paires} $(\lambda_i, p_i)_{i=1}^k$.

%Denote by $$\$$
\end{defn}

\begin{rem}
The homomorphisms in the form above are referred to as diagonal maps in \cite{Vill-sr}, which are different from the diagonal maps considered in this paper.
\end{rem}

\begin{NN}
Let $X_i$ be a sequence of compact Hausdorff spaces, and consider a sequence of diagonal maps $\varphi_i: \mathrm{C}(X_i)\otimes\mathcal K\to\mathrm{C}(X_{i+1})\otimes \mathcal K$. Let $q_1$ be a projection in $\mathrm{C}(X_1)\otimes\mathcal K$, and denote by $q_i=\varphi_{i-1}\circ\cdots\circ\varphi_{1}(q_1)$. Set 
$$A_i=q_i(\mathrm{C}(X_i)\otimes\mathcal K)q_i.$$ The the inductive limit $\varinjlim(A_i, \varphi_i)$ is called a unital AH-algebra with generalized diagonal maps. 

Let $\alpha$ be an open cover of $X_i$. For each $U\in \alpha$, consider any complex valued function $\phi_U$ with $\phi_U(x)\neq 0$ for all $x\in U$ and $\phi_U(x)=0$ for all $x\notin U$. Consider $[U]=[\phi_U \cdot q_i]$ in the Cuntz semigroup of $A_i$. This element is uniquely determined by $U$.% and the Murray-von Neumann class of the restriction of $q_i$ to $U$.

Consider the image $[\varphi_i]([U])$. Then,
$$[\varphi_i]([U])=\sum_{j=1}^{m_i}[(\phi_U\cdot\lambda_j)p_j\otimes q_i]=\sum_{j=1}^{m_i}[(\phi_{\lambda_j^{-1}(U)})p_j\otimes q_i].$$
\end{NN}

\begin{NN}
Let $\alpha$ be an open cover of $X_i$, and let $\varphi_i: A_i\to A_{i+1}$ be a generalized diagonal map. Assume that $X_{i+1}$ is connected and the map $\varphi_i$ is induced by $(\lambda_j, p_j)$. Then, the pair $(\lambda_j, p_j)$ induces a branched cover $$(\lambda, p_j)^{-1}(\alpha):=\{(\lambda^{-1}_j(U), p_j|_{\lambda^{-1}_j(U)});\ U\in\alpha\}$$ of $X_{i+1}$.

Consider the branched cover $$\tilde{\alpha}_{i+1}:=(\lambda_1, p_1)^{-1}(\alpha)\vee(\lambda_2, p_2)^{-1}(\alpha)\vee\cdots\vee(\lambda_{m_i}, p_{m_i})^{-1}(\alpha)$$ and the ordinary cover $${\alpha}_{i+1}:=\lambda_1^{-1}(\alpha)\vee \lambda_2^{-1}(\alpha)\vee\cdots\vee\lambda_{m_i}^{-1}(\alpha).$$

For any $\beta\succ {\alpha}_{i+1}$, consider the branched cover $\mathrm{Ind}_\beta^{\tilde{\alpha}_{i+1}}.$ Set $$n'_{i+1}(\beta):=\mathrm{mul}(\mathrm{Ind}_\beta^{\tilde{\alpha}_{i+1}}),$$ and $$r'_{i+1}(\alpha)=\inf\{\frac{\mathrm{ord}(\beta)}{n_{i+1}(\beta)};\ \beta\succ \alpha_{i+1}\}.$$

Note that $r'_{i+1}$ depends on the pair $(\lambda_j, p_j)$, but the map $\varphi_i$ may be induced by different pairs. Hence one defines $$r_{i+1}(\alpha)=\inf\{r'_{i+1};\ (\lambda_j, p_j)\ \textrm{induces}\ \varphi_i\}.$$

If $X_{i+1}$ has more than one connected component, say, $X_{i+1}^{(1)}, ..., X_{i+1}^{(k)}$, then define $$r_{i+1}(\alpha)=\max\{r_{i+1}(\alpha)\ \textrm{for the restriction of $\varphi_{i}$ to $X_{i+1}^{l}$};\ 1\leq l\leq k \}.$$
\end{NN}

\begin{rem}
Although it is possible to define $r_{i+1}(\alpha)$ without considering connected components of base spaces (direct sum of generalized diagonal maps is a general diagonal map by itself), it would be overkill to compare the overall order of $\beta$ on $X_{j+1}$ to the overall multiplicity of $\beta$. In other words, the order of $\beta$ might be obtained on one connected component, but the multiplicity might be achieved on another.
\end{rem}

\begin{defn}\label{cmdim}
Let $A$ be an AH-algebra with generalized diagonal maps.  Set $$\gamma_i(A):=\sup_{\alpha\in\mathcal{C}(X_i)}\liminf_{j\geq i} r_j(\alpha).$$ The Cuntz mean dimension of $A$ is defined by $$\gamma_\mathrm{c}(A)=\lim_{i\to\infty} \gamma_i(A).$$
\end{defn}

\begin{rem}\label{degenerat}
For an AH-algebra with generalized diagonal maps, if for each diagonal map $\varphi$, all $p_j$ are Murray-von Neumann equivalent, (in particular, if the connecting map $\varphi_i$ is diagonal), then $n_{j+1}(\beta)=\mathrm{rank}(q_{i+1})$ for any $\beta\succ\alpha_{i+1}$, and therefore $$\gamma_i=\frac{\mathcal D(\alpha_{i+1})}{\mathrm{rank}(q_{i+1})}.$$ Hence $\gamma_\mathrm{c}(A)=\gamma(A)$.
\end{rem}

\begin{defn}\label{sep-defn}
Let $X$ and $\Delta$ be two compact Hausdorff spaces. Let $S$ be a unital sub-C*-algebra of $A=p(\mathrm{C}(X)\otimes\mathcal K)p$ and let $\xi: X\to \Delta$ be a continuous map. Then $S$ is said to separate $\Delta$ if 
\begin{enumerate}
\item \label{sep-cond1} for any $y\in \Delta$ and any open set $U\subset \Delta$ containing $y$, there exists $f\in S\cap A'$ such that $f|_{\xi^{-1}(x)}\neq 0$ and $f|_{\xi^{-1}(U^c)}=0$;
\item \label{sep-cond2} for any $x_1$ and $x_2$ in $X$ with $\xi(x_1)=\xi(x_2)$, the representation $\pi_{x_1}|_S$ is unitarily equivalent to $\pi_{x_2}|_S$.
\end{enumerate}
\end{defn}

\begin{lem}\label{sep}
With the notation as above, one then has that for any subset $D\subseteq X$ with a closed image and $x\in X$ with $\xi(x)\notin\xi(D)$, there exists $f\in S\cap A'$ such that $f(x)\neq 0$ but $f|_D=0$.
\end{lem}
\begin{proof}
Set $y=\xi(x)$. Since $\xi(x)\notin \xi(D)$ and $\xi(D)$ is closed, there is an open set $U\subseteq \Delta$ containing $y$, and $f\in S\cap A'$ such that $U\cap\xi(D)=\textrm{\O}$, $f|_{\xi^{-1}(y)}\neq 0$ and $f|_{\xi^{-1}(U^c)}=0$. Note that $D\subseteq \xi^{-1}(U^c)$ and the lemma follows.
\end{proof}

\begin{lem}\label{subset-dim}
Let $X$ be a locally compact second-countable Hausdorff space. If $\mathrm{dim}(C)\leq n$ for any compact subset $C\subseteq X$, then $\mathrm{dim}(C)\leq n+1$.
\end{lem}
\begin{proof}
Denote by $\tilde{X}=X\cup\{\infty\}$ the one-point compactification of $X$. Note that $X$ and $\tilde{X}$ are metrizable. Let us show that $\mathrm{dim}(\tilde{X})\leq n+1$. Then, by Corollary 3.1.20 of \cite{eng-dim}, one has that $\mathrm{dim}(X)\leq n+1$.

Let $\alpha=\{U_1, ..., U_k\}$ be a finite open cover of $\tilde{X}$. Assume that $U_1=\{\infty\}\cup V$ with $V\subseteq X$ and $X\setminus V$ compact. Then $\{U_2, ..., U_k\}$ is a cover of $X\setminus V$. Since $\mathrm{dim}(C)\leq n$ {for any compact subset $C\subseteq X$, one has that $\mathrm{dim}(X\setminus V)\leq n$, and hence there is an open cover $\beta$ of $X\setminus V$} with $\mathrm{ord}(\beta)\leq n+1$ and $\beta$ refines $\{U_2, ..., U_k\}$. Then $\{U_1\}\cup\beta$ is an open cover of $X$ with order at most $n+2$ which also refines $\alpha$. Hence $\mathrm{dim}(\tilde{X})\leq n+1$, as desired.
\end{proof}

\begin{lem}\label{dim-base}
Let $X$ and $\Delta$ be two compact Hausdorff spaces. Let $S$ be a unital sub-C*-algebra of $A=p(\mathrm{C}(X)\otimes\mathcal K)p$ and let $\xi: X\to \Delta$ be a continuous map. Suppose that $S$ separates $\Delta$. Then $S$ has a recursive subhomogeneous C*-algebra decomposition with topological dimension at most $\mathrm{dim}(\Delta)+1$.
\end{lem}
\begin{proof}
Denote by $\mathrm{Prim}(S)$ the space of equivalent classes of the irreducible representations of $S$, and denote by $\mathrm{Prim}_n(S)$ the subspace of the {irreducible representations which have dimension $n$}. Let us first show that {$\mathrm{dim}(\mathrm{Prim}_n(S))\leq\mathrm{dim}(\Delta)+1$} for all $n$. 

For any $\sigma\in\mathrm{Prim}(S)$, there exists $x\in X$ such that $\sigma$ is a direct summand of the evaluation $\pi_x$. (Note that there might be more than one $x$ such that $\pi_x$ contains $\sigma$.) Define the map 
$$F: \textrm{Prim}(S)\ni\sigma\mapsto \xi(x)\in \Delta.$$
(Since $S$ is of type I, one can identify the space of irreducible representations with the space of primitive ideals.)

Assert that $F$ does not depend on the choice of $x$. In fact, let $x_1, x_2\in X$ such that both $\pi_{x_1}|_S$ and $\pi_{x_2}|_S$ contain $\sigma$. If $\xi(x_1)\neq\xi(x_2)$, since $S$ separates $\Delta$, by Lemma \ref{sep}, there exists $f\in S\cap A'$ such that $\pi_{x_1}(f)\neq 0$ but $\pi_{x_2}(f)=0$. In particular, $\pi_{x_1}|_S$ and $\pi_{x_2}|_S$ cannot have a common factor, and thus $\xi(x_1)=\xi(x_2)$.

Note that any pre-image $E$ of a closed subset $D'\subseteq\Delta$ under $F$ has the following form: 
$$E=\{\sigma;\ \textrm{$\sigma$ is a direct summand of $\pi_x|_S$ for some $x\in D$}\},$$
where $D=\xi^{-1}(D')$.

Let us show that $E$ is closed. Denote by $$I(E):=\{f\in S;\ \sigma(f)=0,\ \forall\sigma\in E\}.$$ Let $\sigma'$ be a irreducible representation of $S$ such that $\sigma'(I(E))=\{0\}$. There is $x'\in X$ such that $\sigma'$ is a direct summand of $\pi_{x'}|_S$. If $x'\notin D=\xi^{-1}(D')$, then $\xi(x')\notin D'=\xi(D)$. By Lemma \ref{sep}, there exists $f\in S\cap A'$ such that $f(x')\neq 0$ and $f|_D=0$. Therefore, $f\in I(E)$, but $\sigma'(f)\neq 0$, and this contradicts  the assumption. Thus $x'\in D$, $\sigma'\in E$, and $E$ is a closed set.

Hence the map $F$ is continuous. 

For any $y\in \Delta$, consider the pre-image $F^{-1}(\{y\})$. Let $x\in X$ such that $\xi(x)=y$. Since $S$ separates $\Delta$, by \eqref{sep-cond2} of Definition \ref{sep-defn}, the set $F^{-1}(\{y\})$ agrees with the set of direct summands of $\pi_x|_S$, and therefore $F^{-1}(\{y\})$ only has finitely many elements. 

%For each $1\leq n\leq N$, denote by $F_n$ the restriction of $F$ to $\textrm{Prim}_n(S)$. 
Note that $\textrm{Prim}_n(S)$ is locally compact, Hausdorff, and second-countable, and hence is metrizable. Let $C$ be any compact subset of $\textrm{Prim}_n(S)$. Then the restriction of $F$ to $C$ is a closed mapping. By Theorem 4.3.9,  Theorem 4.1.5, and Corollary 3.1.20 of \cite{eng-dim}, one has that 
$$\mathrm{dim}(C)\leq\mathrm{dim}(F(C))\leq\mathrm{dim}(\Delta).$$
%$$\mathrm{dim}(\mathrm{Prim}_n(S))=\mathrm{dim}(F_n(\mathrm{Prim}_n(S)))\leq\mathrm{dim}(\Delta).$$
By Lemma \ref{subset-dim}, one has that $\mathrm{dim}(\mathrm{Prim}_n(S))\leq\mathrm{dim}(\Delta)+1$

Thus, $S$ is a unital subhomogeneuous C*-algebra with $\mathrm{dim}(\mathrm{Prim}_n(S))\leq\mathrm{dim}(\Delta)+1$ for all $1\leq n\leq N$. By Theorem 2.16 of \cite{Phill-RSA1}, $S$ has a recursive subhomogeneous C*-algebra decomposition with topological dimension at most $\mathrm{dim}(\Delta)+1$, as desired.
\end{proof}

\begin{defn}
Let $\tilde{\alpha}$ be a branched cover of $X$. A family of projections $$\{p^{(i)}_{(U, \kappa_U)};\ (U, \kappa_U)\in\tilde{\alpha}, 1\leq i\leq M_{(U, \kappa_U)}\}$$ is compatible to $\tilde{\alpha}$ if 
\begin{enumerate}
\item each $p^{(i)}_{(U, \kappa_U)}$ is a projection in $\mathrm{C}(\overline{U})\otimes\mathcal K$;
\item for each $U\in\alpha$, the projections $\{p^{(i)}_{(U, \kappa)}; (U, \kappa)\in\tilde{\alpha}, 1\leq i\leq M_{(U, \kappa_U)}\}$ are mutually orthogonal;
%\item for each $U\in\alpha$, one has $\sum_{(U, \kappa)\in\tilde{\alpha}}p_{(U, \kappa)}=$
\item $[p^{(i)}_{(U, \kappa_U)}]=\kappa_U$ for all $1\leq i\leq M_{(U, \kappa_U)}$;
\item\label{coherence} for any $U\cap V\neq\O$, the restrictions of $\{p^{(i)}_{(U, \kappa_U)};\ \kappa, i\}$ to $U\cap V$ are one-to-one corresponding to the restrictions of $\{p^{(i)}_{(V, \kappa_V)};\ \kappa, i\}$ to $U\cap V$.

 %if $[p^{(i)}_{(U, \kappa_U)}|_{U\cap V}]=[p^{(i)}_{(V, \kappa_V)}|_{U\cap V}]$, then $p^{(i)}_{(U, \kappa_U)}|_{U\cap V}=p^{(i)}_{(V, \kappa_V)}|_{U\cap V}$.
\end{enumerate}
\end{defn}

\begin{lem}\label{approx-rsa}
Let $\tilde{\alpha}$ be a branched cover of $X$, and let $\{p_{(U, \kappa_U)}\}$ be a family of projections which is compatible with $\tilde{\alpha}$. Consider the homogeneous C*-algebra $A=p(\mathrm{C}(X)\otimes\mathcal K)p$, where $p$ is a projection in $\mathrm{C}(X)\otimes\mathcal K$. Assume that for each $U\in\alpha$, one has $\sum_{(U, \kappa)\in\tilde{\alpha}}p_{(U, \kappa)}=p|_U$. 

Let $\{\phi_i\}$ is a partition of identity which is subordinate to $\alpha$, then there is a sub-C*-algebra $S\subseteq A$ such that $$\{\phi_U\cdot p_{(U, \kappa_U)};\ (U, \kappa_U)\in\tilde{\alpha}\}\subseteq S,$$ and $S$ has a recursive subhomogeneous decomposition of topological dimension at most $\mathrm{ord}(\alpha)+1$, and the dimensions of the irreducible representations of $S$ are at least $\mathrm{mul}(\tilde{\alpha})$.
\end{lem}

\begin{proof}
Consider a simplicial complex $\Delta$ as the following: The vertices of $\Delta$ consists of elements of $\alpha$. The $s$-dimensional simplices correspond to all $U_1, ..., U_s$ with $\bigcap_{i=1}^s U_i\neq\textrm{\O}$. Let $\{\phi_U; U\in\alpha\}$ be a partition of unit which is subordinate to $\alpha$.  Define the map $\xi: X\to \Delta$ by $$x\mapsto\sum_{U\in\alpha}\phi_U(x)[U].$$ Note that $\mathrm{dim}(\Delta)=\mathrm{ord}(\alpha)$.

For any $U\in\alpha$, consider $\mathcal K_U=\mathcal E_U(1)\cup\cdots\cup\mathcal E_U(n_U)$. Write $$\mathcal E_U(k)=\{\kappa_1, ..., \kappa_{m_k}\}$$ where $1\leq k\leq n_U$. Since $\kappa_i$, $i=1,..., m_k$ are mutually equivalent, there exists a system of matrix units $\{(e^k_U)_{i, j}; 1\leq i, j\leq m_k\}\subseteq A|_U$ such that $[(e^k_U)_{i,i}]=\kappa_i$, $(e^k_U)_{i,i}=p_{(U, \kappa_i)}$. Note that $(e_U^k)_{i, j}\cdot\phi_U\in A$.

Consider the elements $$\mathcal G_{\tilde{\alpha}}:=\{(e^k_U)_{i, j}\cdot\phi_U;\ U\in\alpha, 1\leq k\leq n_U, 1\leq i, j\leq m_k\},$$ and the sub-C*-algebra $$S:=\textrm{C*}(\mathcal G_{\tilde{\alpha}})\subseteq A.$$ Note that for any $(U, \kappa_U)\in\tilde{\alpha}$, 
$$\phi_U\cdot p_{(U, \kappa_U)}=\phi_U\cdot (e_U^k)_{i, i}   \in S.$$

Then, $S$ is a subhomogeneous C*-algebra. For any $x\in X$, the restriction of the representation $\pi_x$ to $S$ has a decomposition 
\begin{equation}\label{lowbd-irr}
\pi_x(S)=\bigoplus_{i}\mathrm{M}_{n_i}(\Comp),
\end{equation} 
with $n_i\geq \mathrm{mul}(\tilde{\alpha})$. 

Let us show that $S$ separates $\Delta$. Note that for each $U\in\alpha$, 
\begin{equation}\label{eqn-cathy}
\phi_U\cdot p=\phi_U\cdot p|_U=\phi_U\cdot (\sum_{(U, \kappa_U)\in\tilde{\alpha}} p_{(U, \kappa_U)})\in S,
\end{equation}
and it is in the centre of $A$. Note that the map $\xi$ induces a homomorphism $$\Xi: \textrm{C}(\Delta)\to\textrm{C}(X)\cong A'.$$ By considering the image of $X$ in $\Delta$, we may assume that $\Xi$ is injective. By \eqref{eqn-cathy}, one has that $\phi_U\cdot p\in\Xi( \textrm{C}(\Delta))$ for any $U\in \alpha$. If $\xi(x_1)\neq\xi(x_2)$ for some $x_1, x_2\in X$, then there exists $\phi_U$ such that $\phi_U(x_1)\neq\phi_U(x_2)$. Therefore, the elements $\{\phi_U\cdot p\}$ separate points of the algebra $\Xi(\textrm{C}(\Delta))$. By the Stone-Weierstrass Theorem, the elements $\{\phi_U\cdot p\}$ generate $\Xi(\textrm{C}(\Delta))$, that is, $S$ contains $\Xi(\textrm{C}(\Delta))$. Therefore, for any $y\in\Delta$ and {any open ball $V\subseteq \Delta$} containing $y$, there exists $f\in S\cap A'$ such that $f|_{\xi^{-1}(y)}=1$ and $f|_{\xi^{-1}(V^c)}=0$. This show the Condition \eqref{sep-cond1} of Definition \ref{sep-defn}.

For \eqref{sep-cond2} of Definition \ref{sep-defn}, if $\xi(x_1)=\xi(x_2)$ for some $x_1, x_2\in X$, then, $\phi_U(x_1)=\phi_U(x_2)$ for all $U\in\alpha$. By Condition \eqref{coherence},  the representation $\pi_{x_1}|_S$ is unitarily equivalent to the representation $\pi_{x_2}|_S$. 

Thus, the subhomogeneous C*-algebra $S$ separates $\Delta$. By Lemma \ref{dim-base}, the C*-algebra $S$ has a recursive subhomogeneous decomposition with topological dimension at most $\mathrm{ord}(\alpha)+1$. Moreover, by \eqref{lowbd-irr}, the irreducible representations of $S$ has dimension at least $\mathrm{mul}(\tilde{\alpha})$.
\end{proof}

\begin{thm}\label{Cuntz-approx}
Let $A$ be a simple AH-algebra with generalized diagonal maps. Then $A$ can be locally approximated by C*-algebras with supreme of their comparison radii at most $\gamma_{\mathrm c}(A)/2$.% $\mathrm{rc}(A)\leq \gamma_{\mathrm c}(A)/2$.
\end{thm}
\begin{proof}
One has to show that for any finite subset $\mathcal F\subseteq A$ and any $\eps>0$, there exists a sub-C*-algebra $S\subseteq A$ such that $\mathcal F\subseteq_\eps S$, and $\mathrm{rc}(S)\leq \gamma_{\mathrm c}(A)/2+\eps.$ %Then, the theorem follows from Lemma \ref{LOA-RC1}.

Without loss of generality, one may assume that $\F\subseteq A_1$. Then there is an open cover $\alpha$ of $X_1$ such that for any $f\in\F$ and any $U\in\alpha$, one has $$\norm{f(x)-f(y)}<\eps\quad\forall x, y\in U.$$

Choose $j$ and an open cover $\beta$ of $X_j$ with $\beta\succ\alpha_j$ such that $$\frac{\mathrm{ord}(\beta)}{n_j(\beta)}<\gamma_\mathrm{c}(A)+\eps.$$

Consider the branched cover $\tilde{\beta}:=\mathrm{Ind}^{\tilde{\alpha}_j}_\beta$. Note that $$\varphi_j(f)=\sum_{i=1}^k(f\circ\lambda_i)p_i,\quad\forall f\in A_1,$$ where $\lambda_i$ are continuous maps from $X_j$ to $X_1$, and $\{p_i;\ 1\leq i\leq k\}$ is a family of mutually orthogonal projections in $\mathrm{C}(X_j)\otimes\mathcal K$. Then, the family of projections $$\{p_i|_{V};\ 1\leq i\leq k, V\in\beta\}$$ is compatible to $\tilde{\beta}$. Indeed, for each projection $p_i|_V$, it corresponds to $(V, [p_i|_V])\in\tilde{\beta}$. 

Let $\{\phi_V;\ V\in\beta\}$ be a partition of identity subordinate to $\beta$, and let $\{x_V;\ V\in\beta\}$ be a set of points with $x_V\in V$. Then, for any $f\in\F$, one has that for any $x\in X_j$,
\begin{eqnarray}
&&\norm{\sum_{i=1}^k\sum_{V\in\beta}(f\circ\lambda_i)(x_V)\phi_V(x)p_i(x)-\varphi_j(f)(x)}\\
&=&\norm{\sum_{i=1}^k\sum_{V\in\beta}(f\circ\lambda_i)(x_V)\phi_V(x)p_i(x)-(\sum_{i=1}^k(f\circ\lambda_i)(x)p_i(x))(\sum_{V\in\beta}\phi_V(x))}\\
&=&\norm{\sum_{i=1}^k\sum_{V\in\beta}(f\circ\lambda_i)(x_V)\phi_V(x)p_i(x)-\sum_{i=1}^k\sum_{V\in\beta}(f\circ\lambda_i)(x)\phi_V(x)p_i(x)}\\
&=&\norm{\sum_{i=1}^k\sum_{V\in\beta}((f\circ\lambda_i)(x_V)-(f\circ\lambda_i)(x))\phi_V(x)p_i(x)}\\
&\leq&\max_{1\leq i\leq k}\sum_{V\in\beta}\norm{((f\circ\lambda_i)(x_V)-(f\circ\lambda_i)(x))\phi_V(x)p_i(x)}\leq\eps,
\end{eqnarray}
and thus $$\varphi_j(\F)\subseteq_\eps \textrm{C*}\{\phi_Vp_i|_V;\ 1\leq i\leq k, V\in\beta\}.$$

By Lemma \ref{approx-rsa}, there is a sub-C*-algebra $S\subseteq A$ such that $$\varphi_j(\F)\subseteq_\eps S,$$ and $S$ has a recursive subhomogeneous C*-algebra decomposition with topological dimension at most $\mathrm{ord}(\beta)+1$, and the dimensions of the irreducible representations of $S$ are at least $n_j=\mathrm{mul}(\tilde{\beta})$.

Since $A$ is simple, $n_j\to\infty$. By Theorem 5.1 of \cite{Toms-Comp-DS}, one has that $\mathrm{rc}(S)\leq \gamma_{\mathrm{c}}(A)/2+\eps$. Thus, $S$ is the desired sub-C*-algebra. %By Lemma \ref{LOA-RC1}, one has that $\mathrm{rc}(A)\leq\gamma_{\mathrm{c}}(A)/2$, as desired.
\end{proof}

\begin{cor}\label{CMD-zero}
Let $A$ be a simple AH-algebra with generalized diagonal maps. If $A$ has Cuntz mean dimension zero, then the Cuntz semigroup of $A$ is almost unperforated. In particular, the C*-algebra $A$ has strict comparison of positive elements.
\end{cor}
\begin{proof}
By Theorem \ref{Cuntz-approx}, for any finite subset $\mathcal F$ and any $\eps>0$, there is a sub-C*-algebra $S$ such that $\mathcal F\subseteq_\eps S$ and $\mathrm{rc}(S)<\eps$. Then, the corollary follows from Lemma \ref{LOA-RC}.
\end{proof}

\begin{cor}\label{CMD-non-zero}
Let $A$ be a simple AH-algebra with generalized diagonal maps. Then one has that $\mathrm{rc}(A)\leq \gamma_{\mathrm c}(A)/2$.
\end{cor}
\begin{proof}
By Theorem \ref{Cuntz-approx}, the C*-algebra $A$ satisfies Condition (1) of Lemma \ref{LOA-RC1}. Since $A$ is simple, the number of eigenvalue maps between each direct summand of $A_i$ and $A_j$ goes to infinity as $j\to\infty$. Thus a simple calculation shows that Condition (2) of Lemma \ref{LOA-RC1} is satisfied automatically. Then, the corollary follows directly from Lemma \ref{LOA-RC1}. 
\end{proof}

\begin{thm}\label{Cuntz-zero}
Let $A$ be a simple AH-algebra with generalized diagonal maps. If $A$ has Cuntz mean dimension zero, then $A$ is isomorphic to an AH-algebra without dimension growth.
\end{thm}
\begin{proof}
By Corollary \ref{CMD-zero} and Lemma \ref{LOA-AD}, the Cuntz semigroup of $A$ is $\mathrm{V}(A)\sqcup\mathrm{SAff}(\tr(A))$. Then, the C*-algebra $A$ is $\mathcal Z$-stable and hence is isomorphic to an AH-algebra without dimension growth.
\end{proof}
One has the following corollary immediately. 
\begin{cor}
Let $A$ be an AH-algebra with generalized diagonal maps with all projections $p_1, ..., p_k$ are equivalent (in particular, if all base spaces of $A_i$ are contractible). If $A$ has at most countably many extreme tracial states, or there exists $M>0$ such that $\rho^{-1}(\kappa)$ has at most $M$ extreme points for any $\kappa\in\textrm{S}(\Kzero(A))$, then $A$ is isomorphic to an AH-algebra without dimension growth.
\end{cor}
\begin{proof}
It follows from Corollary \ref{lisa} and Theorem \ref{RR0-SBP} that $\gamma(A)=0$. By Remark \ref{degenerat}, one has that $\gamma_{\mathrm{c}}(A)=\gamma(A)=0$. Then the AH-algebra $A$ is isomorphic to an AH-algebra without dimension growth by Theorem \ref{Cuntz-zero}.
\end{proof}

This corollary can be generalized as the following:
\begin{defn}
An AH-algebra with generalized diagonal maps has Property (D) if there is $\delta>0$ such that any connecting map can be induced by a pair $(\lambda_i, p_i)$ with a grouping among $\{p_i\}$ $$\{\{p_{1, 1}, ..., p_{1, c_1}\}, ..., \{p_{l, 1}, ..., p_{l, c_l} \}\}$$ such that that any two projection in one group are Murray-von Nuemann equivalent and $$\frac{\sum_{k=1}^{c_j}\textrm{rank}(p_{j, k})}{\sum_{i}\textrm{rank}(p_{i})}\geq\delta$$ for any $j$.
\end{defn}

\begin{cor}
Let $A$ be a simple unital AH-algebra with generalized diagonal maps, {and assume that $A$ has Property (D)}. If $A$ has at most countably many extreme tracial states, or if there exists $M>0$ such that $\rho^{-1}(\kappa)$ has at most $M$ extreme points for any $\kappa\in\textrm{S}(\Kzero(A))$, then $A$ is isomorphic to an AH-algebra without dimension growth.
\end{cor}
\begin{proof}
Suppose that each $X_i$ has only one connected component {(the other cases can be proved in a similar way)}. Note that $\gamma(A)=0$. Fix $A_i$ and consider an open cover $\alpha$ of $X_i$. {Let $\delta$ be the constant of Property (D)}. Fix an arbitrary $\eps>0$. Since $\gamma(A)=0$, one may assume that there is an open cover $\beta$ of $X_{i+1}$ such that $\beta\succ \varphi_{i, i+1}(\alpha)$ and $\textrm{ord}(\beta)<\eps\delta d$, where $d$ is the dimension of the irreducible representation of $A_{i+1}$. 

Consider the induced branched cover by $\beta$. One then has  $$\frac{\mathrm{ord}(\beta)}{n_{i+1}(\beta)}\leq \frac{\mathrm{ord}(\beta)}{d\delta}\leq\eps.$$ Thus, $\gamma_{\textrm{c}}(A)=0$, and hence $A$ is an AH-algebra without dimension growth, as desired.
\end{proof}

\section{Variation mean dimension for general AH-algebras}

For general AH-algebras, the variation mean dimension is considered in this section. Similar to the mean dimension, the class of AH-algebra with variation mean dimension zero are classifiable. However, {it is not clear which global property of the AH-algebra would guarantee the variation mean dimension to be zero. For instance, it is not clear if real rank zero could imply variation mean dimension zero.}

Let $A=\varinjlim(A_i, \varphi_i)$ be an AH-algebra. For any finite subset $\F\subseteq A_i$ and any $\eps>0$, consider $\mathrm{Cov}(\F, \eps)$, the collection of open covers $\alpha$ of $X_i$ satisfying that for any $U\in \alpha$ and any $x, y\in U$, one has that $\norm{f(x)-f(y)}<\eps$, $\forall f\in\F$. Note that if $\alpha\in\mathrm{Cov}(\F, \eps)$ and $\beta\succ\alpha$, then $\beta\in \mathrm{Cov}(\F, \eps)$.

\begin{defn}
For any $\F\in A_i$ and any $\eps>0$, set $$\mathcal{D}(\F, \eps)=\min\{\mathrm{ord}(\alpha);\  \alpha\in\mathrm{Cov}(\F, \eps)\},$$ and set $$\nu_i=\sup_{\F\subseteq A_i, \eps>0}\liminf_{j\to\infty}\max\{\frac{\mathcal{D}(\varphi_{i, j}^k(\F), \eps)}{n_{j, k}};\ 1\leq k\leq h_j\}.$$ The sequence $(\nu_i)$ is increasing, and the limit $\nu$ is the variation mean dimension of $A$.
\end{defn}

\begin{lem}
If $A$ is an AH-algebra with diagonal maps, then $\gamma(A)=\nu(A)$.
\end{lem}
\begin{proof}
For any $\F\subseteq A_i$, $\eps>0$, there is an open cover $\alpha$ of $X_i$ such that for any $U\in\alpha$ and any $x, y\in U$, $\norm{f(x)-f(y)}\leq\eps$ for any $f\in\F$. Since all connecting maps of $A$ are diagonal, one has $$\varphi_{i, j}^k(\alpha)\in\mathrm{Cov}(\varphi_{i, j}^k(\F), \eps),$$ and hence $$\mathcal{D}(\varphi_{i, j}^k(\alpha))\geq\mathcal{D}(\varphi_{i, j}^k(\F), \eps).$$ Therefore, $\gamma(A)\geq\nu(A)$.

%For the simplicity of notation, let us assume that there is only one connected component for each $X_i$ (that is, $h_i=1$). The proof for other cases is similar.

On the other hand, for any $\eps>0$, consider any open cover $\alpha$ of $X_i$, and denote by $d=\mathrm{ord}(\alpha)$. % and choose $\beta\succ\varphi_{i, j}(\alpha)$ such that $$\frac{\mathrm{ord}(\beta^{(k)})}{n_{j}}<\gamma_i+\eps,$$ where $\beta^{(k)}$ is the restriction of $\beta$ to $X_{j, k}$. Note that then $$\frac{\textrm{ord}(\varphi_{j, j'}^{(k)}(\beta))}{n^{(k)}_{j, j'}}<\gamma_i+\eps,\quad\forall j'>j.$$

 %Assume that $\beta\succ \alpha$ and $\mathrm{ord}(\beta)=\mathcal{D}(\alpha)=d$. 
 Pick a partition of unit $\{\psi_U;\ U\in\alpha\}$ subordinate to {$\alpha$}, and regard them as central elements of $A_i$. Set $\F=\{\psi_U;\ U\in\beta\}\subseteq A_i$. 

Assume that there is an open subset $V$ such that $\forall x, y\in V$, $\norm{f(x)-f(y)}< 1/(d+1)$, $\forall f\in\F$. Since $\mathrm{ord}(\alpha)=d$, for any $x\in X_i$, there exists $\psi_U$ such that $\psi_U(x)\geq 1/(d+1)$ (note that $\sum\psi_U(x)=1$).

Fix $x_0\in V$ and $U_0\in {\alpha}$ such that $\psi_{U_0}(x_0)\geq 1/(d+1)$. Since for any $x\in V$, {$$\norm{\psi_{U_0}(x)-\psi_{U_0}(x_0)}< 1/(d+1),$$} one has that $\varphi_{U_0}(x)>0$ and hence $x\in U_0$. Therefore, $V\subseteq U_0$.

Thus, for any open cover $o\in\mathrm{Cov}(\F, 1/(d+1))$, one has that $o\succ{\alpha}$, and hence $$\mathcal{D}(\F, 1/(d+1))\geq \mathcal{D}(\alpha).$$ 

Consider $A_j$ with $j>i$. If there is an open subset $V\subseteq X_j$ such that $\forall x, y\in V$, $\forall f\in\F$, $$\norm{\varphi_{i, j}^k(f)(x)-\varphi_{i, j}^k(f)(y)}< 1/(d+1),$$ then $$\norm{\psi_U\circ\lambda_{i, j}^{l, k}(m)(x)-\psi_U\circ\lambda_{i, j}^{l, k}(m)(y)} < 1/(d+1)$$ for all $U\in \alpha$, $1\leq l\leq h_i$, and $1\leq m\leq m_{i, j}^{l, k}$. 

For any $l$ and $m$, fix $x_0\in V$ and $U_0\in\alpha$ such that $$\psi_{U_0}(\lambda_{i, j}^{l, k}(m)(x_0))\geq 1/(d+1).$$ Hence one has 
$$\lambda_{i, j}^{l, k}(m)(x)\in U_0,\quad \forall x\in V,$$ 
and 
$$V\subseteq (\lambda_{i, j}^{l, k}(m))^{-1}(U_0).$$ 
Therefore $$V\in\varphi_{i, j}^k(\alpha).$$ Thus, for any open cover $o\in\mathrm{Cov}(\varphi_{i, j}^k(\F), 1/(d+1))$, one has that $o\succ{\varphi_{i, j}^k(\alpha)}$, and $$\mathcal{D}(\varphi_{i, j}^k(\F), 1/(d+1))\geq \mathcal{D}(\varphi_{i, j}^k(\alpha)).$$ Hence $\nu(A)\geq\gamma(A),$ as desired.
\end{proof}

\begin{rem}
In general, the mean dimension might not be equal to the variation mean dimension. For example, consider the Villadsen's algebras of second type (\cite{Vill-sr}). These AH-algebras have zero mean dimension zero, but its variation mean dimension is nonzero (Theorem \ref{MD0-Comp-VMD}).
\end{rem}

Using the same argument as that of Theorem \ref{LOA}, one has the following approximation theorem:

\begin{thm}\label{LOA-VMD}
Let $A$ be an AH-algebra, and denote by $\nu$ the variation mean dimension of $A$. Then for any finite subset $\F\subseteq A$ and any $\eps_1>0$ and $\eps_2>0$, there exists a unital sub-C*-algebra $$C\cong\bigoplus p_i\MC{n_i}{\Omega_i}p_i\subseteq A$$ such that $\F\subseteq_{\eps_1} C$, and 
{$$\frac{\mathrm{dim}(\Omega_i)}{\mathrm{rank}(p_i)}<\gamma+\eps_2,\quad\forall i.$$ }
\end{thm}
\begin{proof}
The proof is the same as that of Theorem \ref{LOA} with the following small modification (the same notation as those in the proof of Theorem \ref{LOA} are used): Once the complex $\Delta$ is constructed, consider the map $\xi: X_2\to \Delta$. Define $C$ to be
$$C:=\{f\in A_2;\ \textrm{$f$ is constant on $\xi^{-1}(y)$ for any $y\in\Delta$}\}\subseteq A_2\subseteq A.$$ Then $C$ is the desired sub-C*-algebra.
\end{proof}

Using this approximation theorem, one has the following theorems:

\begin{thm}\label{MD0-Comp-VMD}
If a simple AH-algebra $A$ has variation mean dimension zero, then it has strict comparison of positive elements, and hence is classifiable.
\end{thm}

\begin{thm}
If an AH-algebra $A$ has variation mean dimension $\nu$, and if for any $0\leq s\leq 1$ and any $\eps>0$, there is a projection $p$ such that $\abs{\tau(p)-s}\leq\eps$, $\forall\tau\in\mathrm{T}(A)$, then $\mathrm{rc}(A)\leq \nu/2$.
\end{thm}

\begin{rem}
The variation mean dimension {seems to} depend on the inductive decomposition. Even if one replaces the connecting maps between building blocks by some maps which are unitarily equivalent (this will not change the inductive limit algebra), {it is not clear if the variation mean dimension changes}. 
\end{rem}

\section{A closing remark}

Covering dimension has been generalized to nuclear C*-algebras by W. Winter using order-zero maps. Recall that (see, for example \cite{Winter-Cdim-I})

\begin{defn}
A c.p. map $\psi: A\to B$ is said to have order zero if for any orthogonal elements $a, b\in A^+$, the images $\psi(a)$ and $\psi(b)$ are orthogonal. A c.p. map $\phi: \bigoplus_{i=1}^s \textrm{M}_{r_i}\to A$ is $n$-decomposable if there is a partition $\bigsqcup_{j=1}^n I_j=\{1, ..., s\}$ such that the restriction of $\phi$ to $\bigoplus_{i\in I_j} \textrm{M}_{r_i}$ has order zero for each $j$, and write $\mathrm{dr}(\phi)=n$.
\end{defn}

%For order zero c.p.c. maps with domain a finite dimension C*-algebras, one has that following theorem

\begin{thm}[\cite{Winter-Cdim-I}, Proposition 4.1.1(a)]\label{hom-thm}
If $\phi : F\to A$ is c.p.c. with order zero, then there is a unique *- homomorphism $\pi_\phi: \mathrm{C}F\to A$  such that $\pi_\phi(h\otimes x)=\phi(x)$ $\forall x\in F$, where $\mathrm{C}F$ is the cone $\mathrm{C}_0((0, 1])\otimes F$ over $F$ and $h:=\mathrm{id}_{(0, 1]}$ is the generator of $\mathrm{C}_0((0, 1])$. Conversely, any *-homomorphism $\pi: \mathrm{C}F\to A$ induces such a c.p.c. map $\phi$ of order zero.
\end{thm}

Thus, an $n$-decomposable c.p.c maps with domain finite dimensional C*-algebras are noncommutative analogs for open covers with order $n$.

%Using order-zero c.p.c. maps, the mean dimension for AH-algebra with diagonal maps can be characterized as following.

\begin{defn}\label{defn-nuc-md}
A C*-algebra $A$ is said to have nuclear mean dimension at most $r$ if there is a net $(\psi_\lambda, F_\lambda, \phi_\lambda)$ such that $F_\lambda$ are finite dimensional C*-algebras, $\psi_\lambda: A\to F_\lambda$ are c.p. maps and $\phi_\lambda: F_\lambda\to A$ are c.p.c maps such that 
\begin{enumerate}
\item\label{NMD-cond1} for any $a\in A$, $\phi_\lambda\circ\psi_\lambda(a)\to a$ as $\lambda\to\infty$;
\item\label{NMD-cond2} $$\liminf_{\lambda\to\infty}\frac{\textrm{dr}(\phi_\lambda)}{\textrm{d}(F_\lambda)}\leq r,$$ where $d(F_\lambda)$ is the minimum among the ranks of simple direct summands of $F$.%\item\label{NMD-cond3} for any $\lambda$ and any simple direct summand $M$ of $F_\lambda$, $$[h_{M}, \phi_\lambda(F_\lambda)]=0,$$ where $h_M$ is the positive element $h$ of Theorem \ref{hom-thm} corresponding to the restriction of $\phi_\lambda$ to $M$;
%\item\label{NMD-cond4} there is a choice of sets of matrix units for direct summands of $F$ such that for any two simple direct summands $M$ and $N$ of $F_\lambda$, one has $$h_Mh_N\phi_\lambda(e_{1, 1})=h_Mh_N\phi_\lambda(f_{1, 1}),$$  where $\{e_{i, j}\}$ and $\{f_{i, j}\}$ are the sets of matrix units chosen for  $M$ and $N$.
\end{enumerate}

The C*-algebra $A$ is said to have nuclear mean dimension zero if $r$ can be chosen arbitrarily small.
\end{defn}

\begin{rem}
It is clear that any simple C*-algebra with finite decomposition rank has nuclear mean dimension zero, and it is also a straightforward calculation that any mean dimension zero AH-algebra with diagonal maps has nuclear mean dimension zero. In \cite{Winter-Z-stable-01}, Winter proved that any simple C*-algebras with finite decomposition rank are $\mathcal Z$-stable. So, is any simple nuclear mean dimension zero algebra $\mathcal Z$-stable?
\end{rem}

\bibliographystyle{plain}

%\bibliography{operator_algebras}

\end{document}